\documentclass[10pt]{article} 
\usepackage{cancel}
\usepackage[english]{babel}
\usepackage[utf8x]{inputenc}
\usepackage{amsmath}
\usepackage{amsthm}
\usepackage{graphicx}
\usepackage[title]{appendix}
\usepackage{color}
\usepackage{amsfonts}
\usepackage{epstopdf}
\usepackage{float}
\usepackage{hyperref}
\usepackage[dvipsnames]{xcolor}
\usepackage{caption}
\usepackage{subcaption}
\usepackage{tikz}

\usepackage{amssymb}

\makeindex

\def\dis{\displaystyle}

\def\0vec{\mathbf{0}}

\def\jjnt{\dis\int\!\!\!\!\int}

\def\\Phivec{\mathbf{\Phi}}

\newcommand{\Fin}{\hfill$\Box$}

\def \black {\color{black}}

\newcommand{\R}{\mathbb R}

\newcommand{\iio}{\int_0^T\!\!\!\!\int_{\omega_{0}}}
\newcommand{\iil}{\int_0^T\!\!\!\!\int_{\omega}}
\newcommand{\be}{\begin{equation}}
\newcommand{\ee}{\end{equation}}
\newcommand{\ba}{\begin{eqnarray}}
\newcommand{\ea}{\end{eqnarray}}

\let \dis \displaystyle
\let \ol \overline
\let \bb \mathbb

\let \cal \mathcal

\let \vp \varphi

\title{\bf Exact controllability to the trajectories of the one-phase Stefan problem}

\author{Jon Asier Bárcena-Petisco\thanks{Department\ of Mathematics, University\ of the Basque Country, Barrio Sarriena s/n, 48940 Leioa, Spain. E-mail: {\tt jonasier.barcena@ehu.eus}} \and Enrique Fernández-Cara\thanks{University of Sevilla, Dep.\ EDAN  and IMUS, Univ.\ of Sevilla, Aptdo.~1160, 41080 Sevilla, Spain. E-mail: {\tt cara@us.es}} \and Diego A.  Souza\thanks{University of Sevilla, Dep.\ EDAN  and IMUS, Univ.\ of Sevilla, Aptdo.~1160, 41080 Sevilla, Spain. E-mail: {\tt desouza@us.es}}}

\begin{document}
\date{}

\maketitle

\numberwithin{equation}{section}
\newtheorem{theorem}{Theorem}
\numberwithin{theorem}{section}
\newtheorem{proposition}[theorem]{Proposition}
\newtheorem{conjecture}{Conjecture}
\newtheorem{fact}[theorem]{Fact}
\newtheorem{lemma}[theorem]{Lemma}
\newtheorem{step}{Step}
\newtheorem{corollary}[theorem]{Corollary}
\theoremstyle{remark}
\newtheorem{rmq}[theorem]{Remark}
\newtheorem{definition}[theorem]{Definition}
\newtheorem{example}[theorem]{Example} 
\newtheorem{hypothesis}{Hypothesis}
%\paragraph{}

% ===========================================================
% ABSTRACT
% ===========================================================

\begin{abstract}
        This paper deals with the exact controllability to the trajectories of the one--phase Stefan problem in  one spatial dimension. 
        This is a free-boundary problem that models solidification and melting processes.
        It is assumed that the physical domain is filled by a medium whose state is liquid on the left and solid, with constant temperature, on the right.
        In between we find a free-boundary (the interface that separates the liquid from the solid).
        In the liquid domain, a parabolic equation completed with initial and boundary conditions must be satisfied by the temperature.
        On the interface, an additional free-boundary requirement, called the {\it Stefan condition,} is imposed.
        We prove the local exact controllability to the (smooth) trajectories.
        To this purpose, we first reformulate the problem as the local null controllability of a coupled PDE-ODE system with  distributed controls.
        Then, a new Carleman inequality for the adjoint of the linearized PDE-ODE system, coupled on the boundary through nonlocal in space and memory terms, is presented.
        This leads to the null controllability of an appropriate linear system.
        Finally, a local result is obtained via local inversion, by using {\it Liusternik-Graves' Theorem}. 
        As a byproduct of our approach, we find that some parabolic equations which contains memory terms localized on 
        the boundary are null-controllable.
\end{abstract}

\noindent {\bf Keywords:} Free-boundary problems, one-phase Stefan problem, exact controllability to the trajectories, global Carleman inequalities, Inverse Function Theorem.
%exact controllability.
\vskip 0.25cm\par\noindent
\noindent {\bf Mathematics Subject Classification:} 35R35, 80A22, 93B05, 93C20

\tableofcontents

%%%%%%%%%%%%%%%%%%%%%%%%%%%%%%
%%%%%%%%%%%%%%%%%%%%%%%%%%%%%%
%%%%%%%%%%%%%%%%%%%%%%%%%%%%%%
%%%%%%    INTRODUCTION       %%%%%%%%%%%%
%%%%%%%%%%%%%%%%%%%%%%%%%%%%%%
%%%%%%%%%%%%%%%%%%%%%%%%%%%%%%
%%%%%%%%%%%%%%%%%%%%%%%%%%%%%%

\section{Introduction}

	The phenomena of melting and solidification occurs in a plenty of situations in nature and industry, from melting and freezing of polar ice sheets
	to the continuous casting of steel, {\black see for instance~\cite{saguez}.} The mathematical formulation describing this thermodynamical model of liquid-solid phase transition is known 
	as the {\it Stefan problem}, named after the work of the Slovene physicist and mathematician {\it Josef Stefan}. In such a problem, the model involves a 
	moving free boundary, i.e.\ the spatial physical domain is time-dependent. Physically, in the Stefan problem the dynamics of the liquid-solid
	interface is influenced by the heat flux induced by melting or solidification. Mathematically, the time-evolution of the liquid-solid
	interface is modeled through a nonlinear ordinary differential equation. Among other situations, Stefan problem has also been employed to model 
	population dynamics that describe  {\black tumor growth process~\cite{MR1684873}} and information diffusion in online {\black social networks~\cite{MR2997373}.} 
	 
	 For the sake of completeness, we will give a short description of the mathematical formulation of the Stefan problem. A detailed presentation is given  {\black for instance in~\cite{gupta}.}
	
	Let $\ell_* \in \mathbb{R}_{>0}$ be given.
	At each time $t$, the material domain is separated in two parts:
	the set $(0, \ell(t))$ (the liquid phase domain) and the set $(\ell(t), +\infty)$ (the solid phase domain).
	Here, $\ell = \ell(t)$ indicates the position of the interface;
	it must satisfy $\ell(0) = \ell_0$ and $\ell(t) \in (\ell_*,+\infty)$ at least for all small times, where $\ell_0$ and $\ell_*$ are given and $\ell_0>\ell_*$. 
	Hereafter, for any $T>0$ and any $\ell\in C^0([0,T];\mathbb{R}_{>0})$, we set
$$
	Q_\ell:=\{(x,t): t\in(0,T), \ \ x\in(0,\ell(t))\}~~
	\hbox{and}~~
	H^{1,2}(Q_\ell):=\{u\in L^2(Q_\ell) : u_x,u_{xx},u_t\in L^2(Q_\ell)\}.
$$
	
	This paper deals with the  controllability properties of the following {\it one-phase Stefan problem}:      
\begin{equation}\label{eq:conStefan}
\left\{
\begin{array}{lcl}
u_t-u_{xx}=0& \mbox{ in }& Q_\ell,\\
\noalign{\smallskip}\displaystyle
u(0,t)=v(t)& \mbox{ in }&(0,T),\\
\noalign{\smallskip}\displaystyle
u(\ell(t),t)=0& \mbox{ in }&(0,T),\\
\noalign{\smallskip}\displaystyle
\beta\ell_t(t)=-u_x(\ell(t),t)& \mbox{in} &(0,T),\\
\noalign{\smallskip}\displaystyle
\ell(0)=\ell_0, &&\\
\noalign{\smallskip}\displaystyle
u(\cdot\,,0)=u_0& \mbox{ in }&(0,\ell_0) .
\end{array}
\right.
\end{equation}

	Here, $\beta$ is the so called Stefan number (a positive constant) and the initial state $u_0\in H^1(0,\ell_0)$ satisfies $u_0(x) \geq 0$ for all $x\in[0,\ell_0]$ and $u_0(\ell_0)=0$.
	The functions $u = u(x,t)$ and $v = v(t)$ may be respectively viewed as the temperature of the liquid phase and the imposed temperature on the left.
	In~\eqref{eq:conStefan}, $v$ is the {\it control} (devised for heating or freezing the liquid) and~$(u,\ell)$ is the {\it state.}

	In this paper, the objective is to prove the local exact controllability of \eqref{eq:conStefan} to the (smooth) trajectories at time $T>0$.
	By definition, a trajectory of~\eqref{eq:conStefan} is a triplet  $(\bar u,\bar \ell, \bar v)$ belonging to $H^{1,2}(Q_{\bar \ell})\times H^1(0,T)\times H^{3/4}(0,T)$ satisfying
\begin{equation}\label{eq:freeStefan}
\left\{
\begin{array}{lcl}
\bar u_t-\bar u_{xx}=0& \mbox{ in }& Q_{\bar \ell},\\
\noalign{\smallskip}\displaystyle
\bar u(0,t)=\bar v(t)& \mbox{ in }&(0,T),\\
\noalign{\smallskip}\displaystyle
\bar u(\bar \ell(t),t)=0& \mbox{ in }&(0,T),\\
\noalign{\smallskip}\displaystyle
\beta\bar \ell_t(t)=-\bar u_x(\bar \ell(t),t)& \mbox{in} &(0,T),\\
\noalign{\smallskip}\displaystyle
\bar \ell(0)=\bar \ell_0, &&\\
\noalign{\smallskip}\displaystyle
\bar u(\cdot\,,0)=\bar u_0& \mbox{ in }&(0,\bar \ell_0),
\end{array}
\right.
\end{equation}
	where $\bar \ell_0>\ell_*$, $\bar\ell(t) \in (\ell_*,+\infty)$  for all $t\in[0,T]$, $\bar u_0\in H^1(0,\bar \ell_0)$, $\bar u_0(x)\geq0$ for all $x\in[0,\ell_0]$, $\bar u_0(\ell_0)=0$, $\bar v(t)\geq 0$ for all $t\in[0,T]$ and the compatibility condition $\bar u_0(0)=\bar v(0)$ holds.
	
	We will denote by $\mathcal{T}$ the space of triplets  $(\bar u,\bar \ell, \bar v) \in H^{1,2}(Q_{\bar \ell})\times H^1(0,T)\times H^{3/4}(0,T)$ such that the function $(y,t) \mapsto \bar u(y\bar \ell(t),t)$ belongs to~$W^{1,\infty}(0,T;H^{1}(0,1))$ and~$\bar\ell \in W^{1,\infty}(0,T)$.
	Our main result is the following:
	
\begin{theorem}\label{onephase}
	%Let $T>0$.and $\ell_*>0$ be given,	% and suppose that ${\black \bar u_{xt}\in L^2(Q_{\bar \ell})}$.
	Let $(\bar u,\bar \ell, \bar v)$ be a trajectory of~\eqref{eq:conStefan} {\black with $(\bar u,\bar \ell) \in \mathcal{T}$ and  $\bar v(t)>0$ for all $t\in[0,T]$.}
	Then, there exists $\delta > 0$ with the following property:
	for any $\ell_0\in(\ell_*,+\infty)$ and any $u_0\in H_0^1(0,\ell_0)$ with $u_0(x) \geq 0$ for all $x\in[0,\ell_0]$ satisfying
        \begin{equation}\label{1.2p}
|\ell_0-\bar \ell_0|+\|\ell_0 u_0(\cdot\,\ell_0)-\bar \ell_0 \bar u_0(\cdot\,\bar{\ell}_0)\|_{H^1(0,1)} \leq \delta,
        \end{equation}
   	there exist {\black nonnegative controls $v\in  H^{3/4}(0,T)$} and associated states  $(u,\ell)$ with
\[
     	u\in H^{1,2}(Q_\ell),~\ell\in H^1(0,T)\,~\text{and}\,~\ell(t) \in (\ell_*,+\infty)~~\forall t\in [0,T]
\]
such that 
        \begin{equation}\label{am2}
 \ell(T) = \bar\ell(T)~~\mbox{and}~~ u(\cdot\,,T) = \bar u(\cdot\,,T)~~\mbox{in}~~(0,\bar\ell(T)).
        \end{equation}
\end{theorem}

	Note that the norm in~\eqref{1.2p} is given by
	$$
\|\ell_0 u_0(\cdot\,\ell_0)-\bar \ell_0\bar u_0(\cdot\,\bar \ell_0)\|_{H^1(0,1)} = \left(\int_0^1|\ell_0u_{0}(y\ell_0)-\bar\ell_0\bar u_{0}(y\bar \ell_0)|^2+|\ell_0u_{0,x}(y\ell_0)-\bar\ell_0\bar u_{0,x}(y\bar \ell_0)|^2\,dy\right)^{1/2}.
	$$
	{\begin{rmq}
The assumption $\ol v(t)>0$ on the left edge is a natural assumption as the fluid is liquid, being solid only on the right edge. 
%so the natural thing is that the temperature in the left edge is strictly positive.  	
\end{rmq}
	\begin{rmq} Theorem \ref{onephase} also holds if we just assume that $\bar v$ is nonnegative and $\bar v\not \equiv 0$. 
	\end{rmq}
	\begin{rmq} Thanks to the regularizing effect, Theorem~\ref{onephase} still holds if we only assume that the trajectories belong to~$H^{1,2}(Q_{\bar \ell})\times H^1(0,T)\times H^{3/4}(0,T)$. 
	\end{rmq}
}

	Let us mention some previous works on the control of~\eqref{eq:conStefan} and other similar models.
	
	The analysis of the controllability properties for linear and nonlinear parabolic PDEs in cylindrical parabolic domains is nowadays a classical problem in control theory and some relevant contributions are in~\cite{MR1318622, MR335014, MR1750109, MR1406566, MR1312710} and the references therein. 
	On the other hand, the study of the control and stabilization properties of free-boundary problems for PDEs has not been explored too much, although some important results have been obtained recently;
	see~\cite{MR3772848, MR3433238, MR3654149, MR3993192, MR4317441} and~\cite{MR4032314,kazan}, respectively for one-phase and two-phase Stefan problems.
	In \cite{MR4253800}, the authors study the controllability of free-boundary viscous Burgers equation with one moving endpoint; a similar problem was considered for a 1D fluid-structure system, with modified equations for the interface, see~\cite{MR2139944, MR3736685, MR3023058}.

	In this paper, we are going to consider a different situation, which leads to several new difficulties 
	{and novelties, not found in the previous works on free-boundary problems and fluid-structure models}.
	Let us give more details:	
\begin{itemize}
	\item From our knowledge, our result is the first one concerning the exact control to the trajectories in the context of a parabolic system 
	where the spatial domain changes with time and starts from a different location. 
	\textcolor{black}{Up to now, the available results have dealt with null controllability.
	Indeed, in the context of Stefan problems, the physical meaning of solutions in the previous works is limited due to the fact that 
	the controlled solutions do not preserve positivity. Our findings provide progress in that direction, because our solutions preserve 
	positivity, and thus have a proper physical meaning.}

	\item In fact, we control both components of the state (the final temperature and the final position of the liquid-solid interface).
	This will bring an additional difficulty. 

	\item After a suitable change of variable and some additional arguments, it will be seen that the free-boundary control problem is 
	equivalent to the null controllability of a nonlinear parabolic PDE-ODE system, which can also be viewed as a nonlinear parabolic 
	equation with  {nonlocal and memory} terms on the boundary.
	To establish this property, we will use two main tools:
	a new global Carleman inequality (with weights chosen to handle satisfactorily the boundary terms) and {\it Lyusternik--Graves' Inverse Function Theorem}.
\end{itemize}

	The rest of this paper is organized as follows.
	In~Section~\ref{preliminaries}, we will reformulate the free-boundary problem as a nonlinear parabolic system in a cylindrical domain and we will establish some well-posedness results.
	In~Section~\ref{ECT}, we will present a new Carleman inequality for an adjoint system which leads to the null controllability of a related linearized PDE-ODE system and we will give the proof of~Theorem~\ref{onephase}.
	Finally, the proofs of several results will be presented in~{Appendices~\ref{sec:app_A}, \ref{sec:app_C} and~\ref{sec:app_B}.}

%%%%%%%%%%%%%%%%%%%%%%%%%%%%%%
%%%%%%%%%%%%%%%%%%%%%%%%%%%%%%
%%%%%%%%%%%%%%%%%%%%%%%%%%%%%%
%%%%%%%  PRELIMINARES   %%%%%%%%%%%%%
%%%%%%%%%%%%%%%%%%%%%%%%%%%%%%
%%%%%%%%%%%%%%%%%%%%%%%%%%%%%%
%%%%%%%%%%%%%%%%%%%%%%%%%%%%%%

\section{Preliminaries}\label{preliminaries}

%%%%%%%%%%%%%%%%%%%%%%%%%%%%%%%%%%%%%%%%%%
%%%%%%%%%%%%%%%%%%%%%%%%%%%%%%%%%%%%%%%%%%
%%%%%%%%%%%%%%%%%%%%%%%%%%%%%%%%%%%%%%%%%%
%% REFORMULATION OF THE free-boundary PROBLEM   %%%%%%%%%
%%%%%%%%%%%%%%%%%%%%%%%%%%%%%%%%%%%%%%%%%%
%%%%%%%%%%%%%%%%%%%%%%%%%%%%%%%%%%%%%%%%%%
%%%%%%%%%%%%%%%%%%%%%%%%%%%%%%%%%%%%%%%%%%
	
\subsection{Reformulation of the free-boundary problem in a cylindrical domain} \label{reformulation}

	In order to study the controllability of \eqref{eq:conStefan}, it is convenient to get a reformulation as a nonlinear parabolic equation in a cylindrical domain.
	More precisely, let us set
\[
	p(y,t):=u\left(y\ell(t),t\right) \ \hbox{ and } \ q(t):=\ell(t)^2
\]
	for $(y,t)\in Q_1:=(0,1)\times (0,T)$.
	
	After this transformation, \eqref{eq:conStefan} reads
\begin{equation}\label{eq:conStefantrans}
	\left\{
		\begin{array}{lcl}
			%qp_t-p_{yy}-\frac{1}{2}q'yp_y=0& \mbox{ in }&Q_1,\\
			\dis qp_t-p_{yy}+\frac{y}{\beta}p_y(1,\cdot)p_y=0& \mbox{ in }&Q_1,\\
			p(0,\cdot)= v& \mbox{ in }&(0,T),\\
			p(1,\cdot)=0& \mbox{ in }&(0,T),\\
			p(\cdot\,,0)=p_0& \mbox{ in }&(0,1),\\
			\beta q_t+{2}p_y(1,\cdot)=0& \mbox{ in }& (0,T),\\
			q(0)=q_0, &&
		\end{array}
	\right.
\end{equation}
	where $q_0:=\ell_0^2$ and $p_0(y):=u_0(y\ell_0)$ in $(0,1)$.

\begin{rmq}
  	By introducing the square of $\ell(t)$, the Stefan condition on the interface becomes a linear constraint on~$q_t$ and $p_y(1\,,\cdot)$.
	Otherwise, we would have
\[
	\beta\ell_t(t)=-\frac{1}{\ell(t)}p_y(1,t).
\]
	Since $\ell$ has a strictly positive lower bound $\ell_*$, squaring is  a diffeomorphism. \Fin
\end{rmq}

	With a similar change of variables, \eqref{eq:freeStefan} is transformed into
\begin{equation}\label{eq:conStefanfreetrans}
	\left\{
		\begin{array}{lcl}
                       \dis \bar q\bar p_t-\bar p_{yy}+\frac{y}{\beta}\bar p_y(1,\cdot) \bar p_y=0& \mbox{ in }&Q_1,\\
                        \bar p(0,\cdot)= \bar v& \mbox{ in }&(0,T),\\
                        \bar p(1,\cdot)=0& \mbox{ in }&(0,T),\\
                        \bar p(\cdot\,,0)=\bar p_0& \mbox{ in }&(0,1),\\
                        \beta \bar q_t+{2}\bar p_y(1,\cdot)=0& \mbox{ in }& (0,T),\\
                        \bar q(0)=\bar q_0, &&
		\end{array}
	\right.
\end{equation}
	where $\bar p_0(y):=\bar u_0(y\bar \ell_0)$, $\bar q_0:=\bar{\ell}_0^2$ and $\bar p(y,t)=\bar u\left(\bar\ell(t)y,t\right)$ and $\bar q(t):=\bar \ell(t)^2$
	for $(y,t)\in Q_1$.
	Note that, by assumption, $\bar q(t) \in (q_*,+\infty)$ for all $t\in[0,T]$ with $q_*=\ell_*^2$.
	
	Thus, to prove that \eqref{eq:conStefan} is locally exactly controllable to the trajectory $(\bar u,\bar \ell)$ is equivalent to prove that \eqref{eq:conStefantrans} is locally exactly controllable to~$(\bar p,\bar q)$.
	Consequently, Theorem~\ref{onephase} will be a direct consequence of the following result:
	
\begin{proposition}\label{onephase:fixeddomaincontrol}
	{\black Let $(\bar p,\bar q,\bar v)\in  [W^{1,\infty}(0,T;H^{1}(0,1))\cap H^{1,2}(Q_1)]\times W^{1,\infty}(0,T)\times H^{3/4}(0,T)$ {satisfy}~\eqref{eq:conStefanfreetrans}, 
	with {\black $\bar v(t)>0$ for all $t\in[0,T]$}}.
	Then, there exists $\delta > 0$ with the following property:
	for any $p_0\in H_0^1(0,1)$ and any 
	$q_0\in(q_*,+\infty)$ satisfying 
        \[
	|q_0-\bar q_0|+\|p_0-\bar p_0\|_{H_0^1(0,1)} \leq \delta,
        \]
   	there exist {\black nonnegative controls $v\in  H^{3/4}(0,T)$} and associated solutions $(p,q)$ to~\eqref{eq:conStefantrans}, with
        \[
     	 p\in H^{1,2}(Q_1),~\,q\in H^1(0,T)~\,\text{ and }~\,q(t) \in (q_*,+\infty)~~\forall t\in [0,T] 
        \]
such that 
        \begin{equation}\label{am2}
 q(T) = \bar q(T)~~\mbox{and}~~ p(\cdot\,,T) = \bar p(\cdot\,,T)~~\mbox{in}~~(0,1).
        \end{equation}
\end{proposition}

%%%%%%%%%%%%%%%%%%%%%%%%%%%%%%%%%%%%%%%%%%
%%%%%%%%%%%%%%%%%%%%%%%%%%%%%%%%%%%%%%%%%%
%%%%%%%%%%%%%%%%%%%%%%%%%%%%%%%%%%%%%%%%%%
%% REFORMULATION OF THE ECT - NC  %%%%%%%%%%%%%%%%%%%%
%%%%%%%%%%%%%%%%%%%%%%%%%%%%%%%%%%%%%%%%%%
%%%%%%%%%%%%%%%%%%%%%%%%%%%%%%%%%%%%%%%%%%
%%%%%%%%%%%%%%%%%%%%%%%%%%%%%%%%%%%%%%%%%%

\subsection{Reformulation as a null controllability problem} \label{reformulation2}

	Now, we will reformulate the desired control property as a null controllability problem.
	
	To do this, let us introduce the change of variable $z=p-\bar p$ and ${\black h=\beta (q-\bar q)/2}$.
	Then, the local exact controllability of $(\bar p,\bar q)$ for~\eqref{eq:conStefantrans} is reduced to the local null controllability of the following system, 
	where we have denoted again by~$x$ the spatial variable:
\begin{equation}\label{eq:conStefantransnull}
	\left\{
		\begin{array}{lcl}
			\dis\bar q z_t-z_{xx}+\frac{x}{\beta}\bar p_x(1,\cdot) z_x+\frac{x}{\beta} \bar p_xz_x(1\,,\cdot)+{\black 2\over \beta}\bar p_th + {\black2\over \beta}hz_t+\frac{x}{\beta}z_x(1,\cdot)z_x
			= 0 & \mbox{ in }& Q_1,\\
			z(0,\cdot)= \widehat v& \mbox{ in }&(0,T),\\
			z(1,\cdot)=0& \mbox{ in }&(0,T),\\
			z(\cdot\,,0)=z_0& \mbox{ in }&(0,1),\\
			{\black h_t+z_x(1,\cdot)=0}& \mbox{ in }& (0,T),\\
			h(0)=h_0, &&
		\end{array}
	\right.
\end{equation}
	where  $z_0:=p_0-\bar p_0$, {\black $h_0:=\beta(q_0-\bar q_0)/2$, $\widehat v=v-\bar v$ and $({2/\beta})h(t)+\bar q(t)\in (q_*,+\infty)$} for all $t\in [0,T]$.
	% and the controls $v$ are non-negative functions.}
	Here, we have used \eqref{eq:conStefanfreetrans} to simplify some terms. 

{\black
	Consequently, Proposition~\ref{onephase:fixeddomaincontrol} is obviously equivalent to the following result:}
	
\begin{proposition}\label{onephase:fixeddomaincontrol_nullz}
	Let $(\bar p,\bar q,\bar v)\in  [W^{1,\infty}(0,T;H^{1}(0,1))\cap H^{1,2}(Q_1)]\times W^{1,\infty}(0,T)\times H^{3/4}(0,T)$ {satisfy}~\eqref{eq:conStefanfreetrans}, 
	with $\bar v(t)>0$ for all $t\in[0,T]$.
	There exists $\delta > 0$ with the following property:
	for any $p_0\in H_0^1(0,1)$ and any 
	$q_0\in(q_*,+\infty)$ satisfying 
        \[
	|q_0-\bar q_0|+\|p_0-\bar p_0\|_{H_0^1(0,1)} \leq \delta,
        \]
   	there exist nonnegative controls $v\in  H^{3/4}(0,T)$ and associated solutions $(z,h)$ to~\eqref{eq:conStefantransnull}
	{\black where we have taken $z_0:=p_0-\bar p_0$, $h_0:=\beta(q_0-\bar q_0)/2$ and~$\widehat v=v-\bar v$ such that
        \[
     	 z\in H^{1,2}(Q_1),~\,h\in H^1(0,T)~\,\text{ and }~\,({2/\beta})h(t)+\bar q(t)\in (q_*,+\infty)~~\forall t\in [0,T],
        \]
        \begin{equation}\label{am2}
 h(T) = 0~~\mbox{and}~~ z(\cdot\,,T) = 0~~\mbox{in}~~(0,1).
        \end{equation}
}
\end{proposition}

%%%%%%%%%%%%%%%%%%%%%%%%%%%%%%%%%%%%%%%%%%
%%%%%%%%%%%%%%%%%%%%%%%%%%%%%%%%%%%%%%%%%%
%%%%%%%%%%%%%%%%%%%%%%%%%%%%%%%%%%%%%%%%%%
%% REFORMULATION OF THE ECT - NC  %%%%%%%%%%%%%%%%%%%%
%%%%%%%%%%%%%%%%%%%%%%%%%%%%%%%%%%%%%%%%%%
%%%%%%%%%%%%%%%%%%%%%%%%%%%%%%%%%%%%%%%%%%
%%%%%%%%%%%%%%%%%%%%%%%%%%%%%%%%%%%%%%%%%%

\subsection{Reformulation as a distributed control problem} \label{reformulation2}

	{\black Let us establish a result similar to~Proposition~\ref{onephase:fixeddomaincontrol_nullz} for a distributed control system.}
	%Proposition~\ref{onephase:fixeddomaincontrol_nullz} will then be an immediate consequence.
	
	%Let us present a result in an extended domain for which Proposition~\ref{onephase:fixeddomaincontrol_nullz} will be an immediate consequence. 
	
	Thus, let us set
	$$
Q:=(-1,1)\times (0,T)\ \text{ and } \ H_0^{1,2}(Q) := \{ z \in H^{1,2}(Q) : z(-1,\cdot) = z(1,\cdot) = 0 \ \text{ in } \ (0,T) \}
	$$
and let us consider a non-empty open set $\omega\subset\subset (-1,0)$. 
	The following holds:

\begin{proposition}\label{onephase:fixeddomaincontrol_exdended}
	Assume that $(\bar p,\bar q)\in  [W^{1,\infty}(0,T;H^{1}(-1,1))\cap H^{1,2}_0(Q)]\times W^{1,\infty}(0,T)$, with $\bar q(t)\in (q_*,+\infty)$ for all $t\in [0,T]$.
	There exists $\delta > 0$ with the following property:
	for any $z_0\in H_0^1(-1,1)$ and any $h_0 \in \mathbb{R}$ satisfying 
        \[
	|h_0|+\|z_0\|_{H_0^1(-1,1)} \leq \delta,
        \]
   	there exist controls $w \in  L^2(\omega \times (0,T))$ and associated solutions to the system
 \begin{equation}\label{eq:conStefantransnull_extended}
	\left\{
		\begin{array}{lcl}
			%(h+\bar q)z_t-z_{xx}-\frac{1}{2}(h_t+\bar q') x z_x+\bar p_th+\frac{1}{\beta}\bar p_x x z_x(1\,,\cdot)
			\dis\bar q z_t-z_{xx}+\frac{x}{\beta}\bar p_x(1,\cdot) z_x+\frac{x}{\beta}\bar p_x z_x(1\,,\cdot)+{\black 2\over \beta}\bar p_th + {\black 2\over \beta}hz_t+\frac{x}{\beta}z_x(1,\cdot)z_x
			= w1_\omega & \mbox{ in }& Q,\\
			z(-1,\cdot)= 0& \mbox{ in }&(0,T),\\
			z(1,\cdot)=0& \mbox{ in }&(0,T),\\
			z(\cdot\,,0)=z_0& \mbox{ in }&(-1,1),\\
			{\black h_t+z_x(1,\cdot)=0}& \mbox{ in }& (0,T),\\
			h(0)=h_0 &&
		\end{array}
	\right.
\end{equation}  
with $(z,h)\in H^{1,2}_0(Q)\times H^1(0,T)$, satisfying $\|(z,h)\|_{H^{1,2}_0(Q)\times H^1(0,T)}\leq C\|(z_0,h_0)\|_{ H_0^1(-1,1)\times \mathbb{R}},$ such that 
\begin{equation}\label{am2}
	 \quad h(T) = 0~~\mbox{and}~~ z(\cdot\,,T) = 0~~\mbox{in}~~(-1,1),
\end{equation}
        for some constant $C>0$. 
\end{proposition}
	The proof of Proposition \ref{onephase:fixeddomaincontrol_exdended} will be given in Section \ref{sec:controlnonlinear}.
	The main reason to consider this extended problem is that the boundary controls obtained with the help of Carleman estimates are not sufficiently regular for our purposes, just $L^2(0,T)$, while we need at least~$H^{3/4}(0,T)$.
	On the other hand, with distributed controls, local parabolic results can be used to improve the regularity of the control. 
	
	Obviously, Proposition~\ref{onephase:fixeddomaincontrol_nullz} follows from Proposition~\ref{onephase:fixeddomaincontrol_exdended} by restricting to~$Q_1$ and accepting that the boundary control $\widehat v = \widehat v(t)$ is just the lateral trace of~$z$ at~$x=0$.
	
	 {\black Also, note that we can take $\delta$ small enough to have $({2/\beta})h(t)+\bar q(t)\in (q_*,+\infty)$ for all $t\in [0,T]$.
	 Since $\bar v(t)>0$ for all $t\in[0,T]$, by taking $\delta$ sufficiently small, we can ensure that the control $v:=\widehat v+\bar v$ is nonnegative.}

%%%%%%%%%%%%%%%%%%%%%%%%%%%%%%%%%%%%%%%%%%
%%%%%%%%%%%%%%%%%%%%%%%%%%%%%%%%%%%%%%%%%%
%%%%%%%%%%%%%%%%%%%%%%%%%%%%%%%%%%%%%%%%%%
%%%%%%%%%%%% LINEARIZATION    %%%%%%%%%%%%%%%%%%%%
%%%%%%%%%%%%%%%%%%%%%%%%%%%%%%%%%%%%%%%%%%
%%%%%%%%%%%%%%%%%%%%%%%%%%%%%%%%%%%%%%%%%%
%%%%%%%%%%%%%%%%%%%%%%%%%%%%%%%%%%%%%%%%%%

\subsection{Linearization}

	Now, our aim is to linearize \eqref{eq:conStefantransnull_extended} in a neighborhood of $(0,0)$ and analyze the null controllability properties of the resulting system.
	Thus, let us consider the non-homogeneous linear equation
\begin{equation}\label{eq:conStefantransnulllinear}
\left\{
	\begin{array}{lcl}\dis
		%\bar q z_t- z_{xx}-\frac{1}{2}\bar q' x z_x+ \bar p_t h+\frac{1}{\beta}\bar p_x x z_x(1\,,\cdot)
		\bar q z_t-z_{xx}+\frac{x}{\beta}\bar p_x(1,\cdot) z_x+\frac{x}{\beta}\bar p_x z_x(1\,,\cdot)+{\black 2\over \beta}\bar p_th=f_1+w1_\omega	& \mbox{ in }&Q,\\
 		z(-1,\cdot)= 0																& \mbox{ in }&(0,T),\\
		z(1,\cdot)=0																	& \mbox{ in }&(0,T),\\
		z(\cdot\,,0)= z_0																& \mbox{ in }&(-1,1),\\
		{\black h_t+ z_x(1,\cdot)=f_2}													& \mbox{ in }& (0,T),\\
		h(0)= h_0,&&
	\end{array}
	\right.
\end{equation}
	where $f_1$ and $f_2$ belong to appropriate spaces of functions that decay exponentially as $t\to T^-$ and will be made precise below. 
	
	In order to prove the null controllability of \eqref{eq:conStefantransnulllinear}, we are going to use the Hilbert Uniqueness Method (see \cite{lions1988controlabilite}).
	Accordingly, we will first deduce an observability  inequality for the adjoint, which is the following:
\begin{equation}\label{adj:Stefan}
	\left\{
		\begin{array}{lcl}\dis
			%-\vp_t-B\vp_{xx}+D\vp_x+K\vp=G& \mbox{ in }&Q,\\
			%-(\bar q\vp)_t- \vp_{xx}-\frac{1}{\beta}\bar p_x(1,\cdot)\left(x\vp\right)_x=g_1& \mbox{ in }&Q,\\
			%\noalign{\smallskip}\displaystyle
	-\bar q\vp_t- \vp_{xx}-\frac{x}{\beta}\bar p_x(1,\cdot)\vp_x+\frac{1}{\beta}\bar p_x(1,\cdot)\vp=g_1& \mbox{ in }&Q,\\
			\noalign{\smallskip}\displaystyle
			\vp(-1,\cdot)=0& \mbox{ in }&(0,T),\\
			\noalign{\smallskip}\displaystyle
			\vp(1,\cdot)=\gamma+\int_{-1}^1 \frac{x}{\beta}\bar p_x(x,\cdot) \vp(x,\cdot)\,dx& \mbox{ in }&(0,T),\\
			\noalign{\smallskip}\displaystyle
			\vp(\cdot\,,T)=\vp_T& \mbox{ in }&(-1,1),\\
			\noalign{\smallskip}\displaystyle
			\gamma_t=\int_{-1}^1{\black 2\over \beta}{\bar p_t(x,\cdot)}\vp(x,\cdot)\,dx+g_2& \mbox{ in }& (0,T),\\
			\noalign{\smallskip}\displaystyle
			\gamma(T)=\gamma_T.
		\end{array}
	\right.
\end{equation}

	It is worth to mention that, in  \cite{MR4253800}, the authors
	  point out that the exact controllability to the trajectories for the free-boundary viscous Burgers equation is an open problem. They also 
	  linearize that problem and compute its adjoint system (similar to \eqref{adj:Stefan}).

%%%%%%%%%%%%%%%%%%%%%%%%%%%%%%%%%%%%%%%%%%
%%%%%%%%%%%%%%%%%%%%%%%%%%%%%%%%%%%%%%%%%%
%%%%%%%%%%%%%%%%%%%%%%%%%%%%%%%%%%%%%%%%%%
%% REFORMULATION OF THE free-boundary PROBLEM   %%%%%%%%%
%%%%%%%%%%%%%%%%%%%%%%%%%%%%%%%%%%%%%%%%%%
%%%%%%%%%%%%%%%%%%%%%%%%%%%%%%%%%%%%%%%%%%
%%%%%%%%%%%%%%%%%%%%%%%%%%%%%%%%%%%%%%%%%%

\subsection{Well-posedness of the adjoint system}

	Henceforth, we will denote by  $(\cdot\,,\cdot)_2$ the usual scalar product in $L^2(-1,1)$ and~$\|\cdot\|_2$ will stand for the associated norm.
	
	For clarity, we will provisionally change \eqref{adj:Stefan} by a similar in time system with general coefficients:
\begin{equation}\label{adj:Stefanplain}
\left\{
	\begin{array}{lll}
	-\bar q(t)\varphi_t-\varphi_{xx}-a\varphi_x-b\varphi=f& \mbox{ in }&Q,\\
	\noalign{\smallskip}\displaystyle
	\varphi(-1,\cdot)=0& \mbox{ in }&(0,T),\\
	\noalign{\smallskip}\displaystyle
	\varphi(1,t)=\gamma(t)+(N(\cdot\,,t),\varphi(\cdot\,,t))_2& \mbox{ in }&(0,T),\\
	\noalign{\smallskip}\displaystyle
	\varphi(\cdot\,,T)=\varphi_T& \mbox{ in }&(-1,1),\\
	\noalign{\smallskip}\displaystyle
	\gamma'(t)=(R(\cdot\,,t),\varphi(\cdot\,,t))_2+g(t)& \mbox{ in }&(0,T),\\
	\noalign{\smallskip}\displaystyle
	\gamma(T)=\gamma_T.
\end{array}
\right.
\end{equation}
%	Here, we assume that $\ol q\in C^1([0,T];\mathbb{R}_+)$ such that:
%\begin{equation}\label{eq:minS}
%	\min_{[0,T]}\ol q>0.
%\end{equation}
%	Moreover, we assume that $R\in L^\infty(Q)$ and 
%$N\in W^{1,\infty}(0,T;L^\infty)$. 
	Note that the boundary condition on $\varphi$ at $x=1$ involves $\gamma$, that is essentially a primitive in time of a spatial integral of $\varphi$
	and an additional spatial integral of $\varphi$. Thus, in this system, we find nonlocal in space and memory boundary terms.

	The following result holds:
\begin{proposition}\label{prop:sol_adj}
	Let us assume that $R\in L^2(Q)$, $N\in H^1(0,T;L^2(-1,1))$, $a,\,b\in  L^2(0,T;L^\infty(-1,1))$ and ${\black\bar q\in C^0([0,T])}$ with $\bar q(t) \in(q_*,+\infty)$ for all $t\in [0,T]$. 
	Let $f\in L^2(Q)$, $g\in L^2(0,T)$, $\varphi_T\in H^1(-1,1)$ and $\gamma_T \in \R$ be given and assume that
\begin{equation}\label{compatib_cond}
 	\varphi_T(-1)=0~\,\hbox{and}~\,\varphi_T(1)=\gamma_T+(N(\cdot\,,T),\varphi_T)_2.
\end{equation}
	Then, there exists a unique strong solution in $H^{1,2}(Q)\times H^1(0,T)$ to  \eqref{adj:Stefanplain} such that the following estimate holds:
	\begin{equation}\label{est:regadj}
\|\varphi\|_{ H^{1,2}(Q)}^2+\|\gamma\|_{H^1(0,T)}^2
\leq e^{C(1+T)} \left(
\|f\|_{L^2(Q)}^2+\|g\|_{L^2(0,T)}^2+\|\varphi_0\|_{H^1(0,1)}^2+|\gamma_0|^2\right),
	\end{equation}
where $C$ is a positive constant depending on $a$, $b$, $R$, $N$ and $\bar q$ but independent of $T$.
\end{proposition}

	The proof is given in Appendix~\ref{sec:app_A}.

%%%%%%%%%%%%%%%%%%%%%%%%%%%%%%%%%%%%%%%%%%
%%%%%%%%%%%%%%%%%%%%%%%%%%%%%%%%%%%%%%%%%%
%%%%%%%%%%%%%%%%%%%%%%%%%%%%%%%%%%%%%%%%%%
%% REFORMULATION OF THE free-boundary PROBLEM   %%%%%%%%%
%%%%%%%%%%%%%%%%%%%%%%%%%%%%%%%%%%%%%%%%%%
%%%%%%%%%%%%%%%%%%%%%%%%%%%%%%%%%%%%%%%%%%
%%%%%%%%%%%%%%%%%%%%%%%%%%%%%%%%%%%%%%%%%%

	\subsection{Carleman estimates for parabolic equations with nonlocal boundary conditions}
	
Let us recall the definition of several classical weights, frequently used in connection with global Carleman inequalities for parabolic equations, see \cite{MR1406566}.
	
	Let $\omega_0$ be a non-empty open set, with~$\omega_0 \subset\subset \omega$ and let be a function $\eta$ in~$C^2([-1,1])$ satisfying
\begin{equation}\label{P1}
    		\eta>0\ \hbox{ in } \ [-1,1], %\ \  \eta_{xx}<0 \ \hbox{ in } \ [-1,1],
		\ \ \min_{x \in [-1,1]\setminus\omega_0}|\eta_x(x)|>0,
%		 \mbox{ in }[-1,1]\setminus\omega_0
    		\ \ \eta(-1)=\eta(1)=\dis\min_{x\in[-1,1]}\eta(x).
\end{equation}

	Let us introduce the following associated weights:
\[
	\begin{array}{ll}\dis
	\alpha (x,t) :={e^{2\lambda m \|\eta \|_\infty} - e^{\lambda(m\|\eta\|_\infty+\eta (x))}\over t(T-t)}&\forall(x,t) \in Q,\\
	\noalign{\smallskip}\dis
    	\xi (x,t) := {e^{ \lambda(m\|\eta\|_\infty+\eta (x))}\over t(T-t)}&\forall(x,t)\in Q,\\
	%\alpha^*(t):= \min\limits_{x\in [0,L]}\alpha(x,t),%={e^{2\lambda m \|\eta \|_\infty} - e^{\lambda(m\|\eta\|_\infty+\eta (0))}\over t(T-t)}=\alpha(0,t),
	%&\forall t\in (0,T),\\
	%\noalign{\smallskip}\dis
	%\xi^*(t) := \max\limits_{x\in [0,L]}\xi(x,t), %={ e^{\lambda(m\|\eta\|_\infty+\eta (0))}\over t(T-t)}=\xi(0,t),
	%&\forall t\in (0,T),\\
	\noalign{\smallskip}\dis
	\widehat{\alpha}(t):= \max\limits_{x\in [-1,1]}\alpha(x,t)%={e^{2\lambda m \|\eta \|_\infty} - e^{\lambda(m\|\eta\|_\infty+\eta (1))}\over t(T-t)}
	=\alpha(1,t)=\alpha(-1,t)&
	\forall t\in (0,T),\\
	\noalign{\smallskip}\dis
		\widehat{\xi}(t) := \min\limits_{x\in [-1,1]}\xi(x,t)=%{e^{\lambda(m\|\eta\|_\infty+\eta (1))}\over t(T-t)}
		\xi(1,t)=\xi(-1,t)&\forall t\in (0,T),
\end{array}
\]	
	where $m>1$ and $\lambda >0$ is a sufficiently large constant (to be chosen later).

    We will establish a Carleman inequality that holds for the solutions to a simplified version of \eqref{adj:Stefanplain}. This will be later extended to the solutions to \eqref{adj:Stefanplain} and, consequently, to the adjoint states in \eqref{adj:Stefan}.
\begin{lemma}\label{lemma:nonlocal}
	Let us assume that $R\in L^\infty(0,T;L^2(-1,1))$, $N\in W^{1,\infty}(0,T;L^2(-1,1))$ and $d\in C^1([0,T])$ with $d(t)>d_*>0$ for all $t\in [0,T]$. 
	There exist constants  $\lambda _0\geq1$, $s_0\geq1$ and $C_0>0$ such that, for any~$\lambda \geq \lambda _0$, any~$s\ge s_0(T+T^2)$, 
	any~$(\psi_T,\gamma_T)\in H^1(-1,1)\times  \mathbb{R}$ satisfying \eqref{compatib_cond} and any source terms $f\in L^2(Q)$ and  $g\in L^2(0,T)$, the strong solution to
\begin{equation}\label{eq:adj_1}
    \left \{
        \begin{array}{lcl}
             \psi_t+d(t)\psi_{xx} =f &\mbox{in}&    Q,\\
            \noalign{\smallskip}\dis
	\psi(-1,\cdot)=0 &\mbox{in}&    (0,T),\\
    \noalign{\smallskip}\dis
	\psi(1,t) =\gamma(t)+(N(\cdot\,,t),\psi(\cdot\,,t))_2 &\mbox{in}&    (0,T),\\
	\noalign{\smallskip}\dis
	\psi(\cdot\,,T) = \psi_T &\mbox{in}&    (-1,1)\\
	\noalign{\smallskip}\dis
	\gamma_t(t) -(R(\cdot\,,t),\psi(\cdot\,,t))_2=g&\mbox{in}&    (0,T),\\
        	\noalign{\smallskip}\dis
       	\gamma(T) = \gamma_T &&     
        \end{array}
    \right.
\end{equation}
    satisfies
\begin{equation}\label{carleman:1}
	\begin{alignedat}{2}
    		&\jjnt_Q \left[(s\xi)^{-1}(|\psi_{xx}|^2+|\psi_t|^2)+\lambda^2(s\xi)|\psi_x|^2 +\lambda^4(s\xi)^3|\psi|^2\right]e^{-2s\alpha}\,dx\,dt\\
    		&\quad+\int_0^T\left[\lambda^3(s\widehat\xi)^3 |\psi(1,t)|^2 +\lambda (s\widehat\xi)(|\psi_x(-1,t)|^2
		+|\psi_x(1,t)|^2)\right]e^{-2s\widehat\alpha}\,dt  \\
    		& \leq C_0 \left(s^3\lambda^4\iil\xi^3 |\psi|^2 e^{-2s\alpha}\,dx\,dt+\jjnt_Q |f|^2e^{-2s \alpha}\,dx\,dt
		+\int_0^T|g|^2e^{-2s\widehat\alpha}\,dt\right).    
		%+s^{-1}\lambda^{-1}\int_0^T\widehat\xi^{-1}e^{-2s\widehat\alpha}|g|^2\,dt\right).    
     \end{alignedat}
	\end{equation}
\end{lemma}

	Note that, in view of~$\eqref{eq:adj_1}_3$ and~$\eqref{eq:adj_1}_5$, we can also include weighted integrals of~$\gamma$ and~$\gamma_t$ in the left hand side of~\eqref{carleman:1}.
	
	The proof of Lemma~\ref{lemma:nonlocal}  is given in~Appendix~\ref{sec:app_C}.
	As already mentioned, this Carleman inequality is new.
	It is one of the main contributions in the paper. The main difficulty to overcome is that we have to deal with a non-local term on the boundary, both in the space and time variables.

\subsection{Well-posedness of the linearized system}

	%In the sequel, $(\cdot\,,\cdot)_2$ denotes the usual scalar product in $L^2(-1,1)$ and~$\|\cdot\|_2$ stands for the associated norm.
	The aim of this section is to prove the existence and uniqueness of a global solution to~\eqref{eq:conStefantransnulllinear}.

	For convenience, we will establish the existence and uniqueness of a strong solution to a similar, where (again) we have introduced general coefficients.

	More precisely, we have the following result:
	
\begin{proposition}\label{prop:sol_adj_7}
	Assume that $(a,R,N)$ belongs to the space~$L^2(0,T;L^{\infty}(-1,1))\times L^2(Q)\times L^\infty(0,T;L^2(-1,1))$ and~${\black\ol q\in W^{1,\infty}(0,T)}$, with $\ol q(t)\in (q_*,+\infty)$ for all $t\in [0,T]$. 
	Let $F\in L^2(Q)$, $G\in L^2(0,T)$, $z_0\in H^1_0(-1,1)$ and $h_0\in \mathbb{R}$ be given.
	There exists a unique strong solution in~$H^{1,2}_{0}(Q)\times H^1(0,T)$ to the system
\begin{equation}\label{eq:conStefantransnulllinear1}
\left\{
	\begin{array}{lcl}
		\bar q (t) z_t- z_{xx}+az_x+ Rh+Nz_x(1,\cdot)=F	& \mbox{ in }&Q,\\
 		z(-1,\cdot)= 0					& \mbox{ in }&(0,T),\\
		z(1,\cdot)=0					& \mbox{ in }&(0,T),\\
		z(\cdot\,,0)= z_0				& \mbox{ in }&(-1,1),\\
		h_t+ z_x(1,\cdot)=G		& \mbox{ in }& (0,T),\\
		h(0)= h_0, &&
	\end{array}
	\right.
\end{equation}
	such that the following inequality holds:
\begin{equation}\label{est:regadj}
\|z\|_{ H^{1,2}_{0}(Q)}^2+\|h\|_{H^1(0,T)}^2
\leq e^{C(1+T)}\left(
\|F\|_{L^2(Q)}^2+\|G\|_{L^2(0,T)}^2+\|z_0\|_{H^1_0(-1,1)}^2+|h_0|^2\right),
\end{equation}
where $C$ is a positive constant  depending on $a$, $R$, $N$ and $\ol q$ but independent of $T$.
\end{proposition}

	The proof is given in Appendix~\ref{sec:app_B}.
	
	At this point, we will introduce the definition of {\it solution by transposition} to~\eqref{eq:conStefantransnulllinear1}:
	
\begin{definition}\label{weak_heat_transposition}
	It will be said that $(z,h)\in L^2(Q)\times L^2(0,T)$ is a solution by transposition to~\eqref{eq:conStefantransnulllinear1} if
 \begin{equation}\label{eq:gunogdos}
 	\jjnt_Q z(x,t)f(x,t)\,dx\,dt+\int_0^T h(t) g(t)\,dt= M(f,g)\quad\forall (f,g)\in L^2(Q)\times L^2(0,T),
\end{equation}
	where the linear form $M:L^2(Q)\times L^2(0,T) \mapsto \mathbb{R}$ is given by
$$
%\cal L^*(\phi,\gamma)=(-(\ol s\vp)_t-\vp_{xx}+\frac{\ol{s}'}{2}(x\vp)_x), \int_0^1\ol w_t(x,t)\vp(x,t)\,dx-\beta\gamma')
	M(f, g):=\jjnt_Q F(x,t)\varphi(x,t)\,dx\,dt+\bar q(0)(z_0,\varphi(\cdot\,,0))_2
	+h_0\gamma(0)+\int_0^T G(t)\gamma(t)\,dt
$$
%\ol{s}(0)\int_0^1 z_0(x)\cdot\phi(x,0)\,dx+\beta h_0\kappa(0)+\jjnt_Q f_1\phi+\int_0^Tf_2\kappa
	and $(\varphi,\gamma)$ is the unique strong solution to
\begin{equation}\label{adj:Stefan_transposition}
	\left\{
		\begin{array}{lcl}
			%-\vp_t-B\vp_{xx}+D\vp_x+K\vp=G& \mbox{ in }&Q,\\
			-(\bar q\varphi)_t- \varphi_{xx}-\left(a\varphi\right)_x=f& \mbox{ in }&Q,\\
			\noalign{\smallskip}\displaystyle
			\varphi(-1,\cdot)=0& \mbox{ in }&(0,T),\\
			\noalign{\smallskip}\displaystyle
			\varphi(1,t)=\gamma(t)+(N(\cdot\,,t), \varphi(\cdot\,,t))_2& \mbox{ in }&(0,T),\\
			\noalign{\smallskip}\displaystyle
			\varphi(\cdot\,,T)=0& \mbox{ in }&(-1,1),\\
			\noalign{\smallskip}\displaystyle
			\gamma'(t)=(R(\cdot\,,t), \varphi(\cdot\,,t))_2+g& \mbox{ in }& (0,T),\\
			\noalign{\smallskip}\displaystyle
			\gamma(T)=0.
		\end{array}
	\right.
\end{equation}
\end{definition}

	Since the boundary and final conditions in~\eqref{adj:Stefan_transposition} satisfy the appropiate compatibility conditions \eqref{compatib_cond}, Proposition~\ref{prop:sol_adj} guarantees the existence and uniqueness of a strong solution to \eqref{adj:Stefan_transposition}.
	Consequently, Definition \ref{weak_heat_transposition} makes sense.

%\begin{rmq}[Relation between strong solution and solution by transposition]
%\label{rq:strgsoltrassol}
%A strong solution of \eqref{eq:conStefantransnulllinear} is also a solution by transposition, which can be proved by integration by parts. 
%\end{rmq}

\begin{proposition}\label{ex_un_mixed_sol}
	{Let the assumptions in~Proposition~\ref{prop:sol_adj_7} be satisfied and, additionally, consider that
	$a\in L^2(0,T;W^{1,\infty}(-1,1))$ and $N\in H^1(0,T;L^2(-1,1))$. Then, there exists a unique solution by transposition to~\eqref{eq:conStefantransnulllinear1}. }
\end{proposition}

\begin{proof}
	Note that  $M:L^2(Q)\times L^2(0,T) \mapsto \mathbb{R}$ is a continuous linear form.
	Indeed, since $(\varphi,\gamma)$ is the unique strong solution, from Proposition~\ref{prop:sol_adj} we have
\[
	\|\varphi\|_{ H^{1,2}(Q)}^2+\|\gamma\|_{H^1(0,T)}^2\leq
	 C \|(f,g)\|^2_{L^2(Q)\times L^2(0,T)}.	
\]
	Therefore, we deduce from {\it Riesz Representation Theorem} that there exists exactly one  solution by transposition to~\eqref{eq:conStefantransnulllinear1}. 
\end{proof}

	Notice that strong solutions to \eqref{eq:conStefantransnulllinear1} are solutions by transposition.

%%%%%%%%%%%%%%%%%%%%%%%%%%%%%%%%%%%%%%%%%%
%%%%%%%%%%%%%%%%%%%%%%%%%%%%%%%%%%%%%%%%%%
%%%%%%%%%%%%%%%%%%%%%%%%%%%%%%%%%%%%%%%%%%
%% REFORMULATION OF THE free-boundary PROBLEM   %%%%%%%%%
%%%%%%%%%%%%%%%%%%%%%%%%%%%%%%%%%%%%%%%%%%
%%%%%%%%%%%%%%%%%%%%%%%%%%%%%%%%%%%%%%%%%%
%%%%%%%%%%%%%%%%%%%%%%%%%%%%%%%%%%%%%%%%%%

\section{Exact controllability to the trajectories}\label{ECT}

	This section is devoted to prove the null controllability of the linear system~\eqref{eq:conStefantransnulllinear} and the local null controllability of the nonlinear PDE-ODE system~\eqref{eq:conStefantransnull_extended}.

%%%%%%%%%%%%%%%%%%%%%%%%%%%%%%%%%%%%%%%%%%
%%%%%%%%%%%%%%%%%%%%%%%%%%%%%%%%%%%%%%%%%%
%%%%%%%%%%%%%%%%%%%%%%%%%%%%%%%%%%%%%%%%%%
%% REFORMULATION OF THE free-boundary PROBLEM   %%%%%%%%%
%%%%%%%%%%%%%%%%%%%%%%%%%%%%%%%%%%%%%%%%%%
%%%%%%%%%%%%%%%%%%%%%%%%%%%%%%%%%%%%%%%%%%
%%%%%%%%%%%%%%%%%%%%%%%%%%%%%%%%%%%%%%%%%%

\subsection{Controllability of the linearized problem}

   	We will present a suitable Carleman inequality for the solutions to a properly chosen adjoint system. This will imply 
   	the null controllability of the linearized system  \eqref{eq:conStefantransnulllinear} (see Proposition~\ref{prop:contlin} below). 
   	This result will be essential for the proof of Proposition \ref{onephase:fixeddomaincontrol_exdended} (the local null controllability of~\eqref{eq:conStefantransnull_extended}.

\subsubsection{A Carleman inequality}

%	We will present a suitable new Carleman inequality for the adjoint system \eqref{adj:Stefan}.
%	As we will see the null controllability for system \eqref{eq:conStefantransnulllinear}  will follow from a such Carleman inequality.
%	In order to prove this  Carleman inequality we will use Lemma \ref{lemma:nonlocal}. More precisely, we have:

	The following holds:
\begin{theorem}\label{cl:carlemanadj}
	%Let the assumptions in~Proposition~\ref{lemma:nonlocal} be satisfied.
	Assume that~$(\bar p,\bar q)$ belong to the space~$[W^{1,\infty}(0,T;H^{1}(-1,1))\cap H^{1,2}_0(Q)]\times W^{1,\infty}(0,T)$ with $\ol q(t)\in (q_*,+\infty)$ for all $t\in [0,T]$.
	There exist constants  $\lambda _0\geq1$,  $s_0\geq1$ and $C_0>0$  such that,  for any $\lambda \ge \lambda _0$, any~$s\ge s_0(T+T^2)$, any $\vp_T\in H^1(-1,1)$  any~$\gamma_T\in \mathbb{R}$ with
\begin{equation}\label{compatib_cond_2}
 \vp_T(-1)=0\quad \hbox{and}\quad\vp_T(1)=2\gamma_T + \frac{1}{\beta}\int_{-1}^1 \bar p_x(x,T) x \vp_T(x)\,dx
\end{equation}
	and any right hand sides $g_1\in L^2(Q)$ and $g_2\in L^2(0,T)$, the strong solution to~\eqref{adj:Stefan} satisfies:
	\begin{equation}\label{carleman:2}
\begin{array}{l} \dis
\!\!\!\!\jjnt_Q \left[(s\xi)^{-1}(|\vp_t|^2 + |\vp_{xx}|^2)+\lambda^2(s\xi)|\vp_x|^2 
+\lambda^4(s\xi)^3 |\vp|^2\right]e^{-2s\alpha}\,dx\,dt \\
\dis+\int_0^T\left[ |\gamma_t|^2+\lambda (s\widehat\xi)\left(|\varphi_x(-1,t)|^2
		+|\varphi_x(1,t)|^2\right)+\lambda^3(s\widehat\xi)^3\left(|\varphi(1,t)|^2+|\gamma|^2\right)\right]e^{-2s\widehat\alpha}\,dt
\\ \quad \dis
    %  +\int_0^T (\lambda^3(s\widehat\xi)^3|\vp(L,t)|^2+
    %\lambda (s\widehat\xi) |\vp_x(L,t)|^2))e^{-2s\widehat\alpha}\,dt +
%\int_0^T\lambda \left((s\widehat\xi)^3 |\gamma|^2+|\gamma'|^2\right)
%e^{-2s\alpha}\,dt \\
\leq C_0 \left(\jjnt_Q|g_1|^2 e^{-2s \alpha}\,dx\,dt
+ \int_0^T |g_2|^2e^{-2s\widehat\alpha}\,dt+s^3\lambda^4\iil  \xi^3|\vp|^2e^{-2s \alpha}\,dx\,dt \right).
\end{array}
	\end{equation}
\end{theorem}

\begin{proof}
	Let us apply~Lemma~\ref{lemma:nonlocal} with the following data:
$$
d={1\over \bar q},\ \ f=-{1\over\bar q}\left[g_1+\frac{\bar p_x(1,\cdot)}{\beta}(x\vp_x-\vp)\right], \ \ N(x,t)= \frac{x}{\beta}  \bar p_x(x,t), \ \ R={2\over \beta}\bar p_t \ \text{ and } \ g=g_2.
$$
	We obtain:
\[
	\begin{alignedat}{2}
&\!\!\!\!\jjnt_Q\!\! \left[(s\xi)^{-1}(|\vp_t|^2 + |\vp_{xx}|^2)+\lambda^2(s\xi)|\vp_x|^2 
+\lambda^4(s\xi)^3 |\vp|^2\right]e^{-2s\alpha}\,dx\,dt\\
&\dis+\int_0^T\left[|\gamma_t|^2 +\lambda (s\widehat\xi)\left(|\varphi_x(-1,t)|^2
		+|\varphi_x(1,t)|^2\right)+\lambda^3(s\widehat\xi)^3 \left(|\gamma|^2+ |\varphi(1,t)|^2\right)\right]e^{-2s\widehat\alpha}\,dt \\ \quad \dis
    		& \leq C_0 \left(s^3\lambda^4\iil\xi^3 |\varphi|^2 e^{-2s\alpha}\,dx\,dt+\jjnt_Q |f|^2e^{-2s \alpha}\,dx\,dt
		+\int_0^T|g|^2e^{-2s\widehat\alpha}\,dt\right).    
		%+s^{-1}\lambda^{-1}\int_0^T\widehat\xi^{-1}e^{-2s\widehat\alpha}|g|^2\,dt\right).    
     \end{alignedat}
\]

	Clearly, one can absorb the lower order terms in $f$ and obtain \eqref{carleman:2}.
	% and then use $\eqref{adj:Stefan}_3$--$\eqref{adj:Stefan}_5$ to obtain the estimates for~$\gamma$ and~$\gamma_t$.
\end{proof}

\subsubsection{Null controllability with nonhomogeneities}

	In this section we prove the null controllability property of \eqref{eq:conStefantransnulllinear}  with source terms that decay exponentially as $t\to T^-$.
	As we will see below, this result will be useful to prove the local null controllability of \eqref{eq:conStefantransnull}.
	
	Before, it will be convenient to deduce a second Carleman inequality with weights that do not vanish at $t = 0$.
	
	More precisely, let the function $r = r(t)$ be given by
\begin{equation}\label{def:l(t)}
r(t)=\begin{cases}
T^2/4 \ \ &  \mbox{ in } [0,T/2], \\
t(T-t) \ \ & \mbox{ in } [T/2,T]
\end{cases}
\end{equation}
	and let us set $D_1=(-1,1)\times (0,T/2)$, $D_2=(-1,1)\times (T/2,T)$,
\begin{equation}\label{eq:weightsbeta}
\zeta (x,t):=\frac{e^{2\lambda m \|\eta \|_\infty} - e^{\lambda(m\|\eta\|_\infty+\eta (x))}}{r(t)} \quad\hbox{and}\quad \mu (x,t):= \frac{e^{\lambda(m\|\eta\|_\infty+\eta (x))}}{r(t)} \quad \forall (x,t)\in Q.
\end{equation}
	Let us also introduce the notation: 
$$
%{(T^2l(t/T))^m},
 \widehat\zeta(t):=\!\!\max_{x\in [-1,1]}\!\!\zeta(x,t),~ \widehat\mu(t):=\!\!\min_{x\in [-1,1]}\!\mu(x,t),~ \zeta^*(t):=\!\!\min_{x\in [-1,1]}\!\!\zeta(x,t),~ \mu^*(t):=\!\!\max_{x\in [-1,1]}\!\mu(x,t)\quad \forall t\in(0,T)
$$
	and
$$
	\rho_0(t):=e^{s\zeta^*(t)},~ \rho_1(t):=  e^{s\widehat\zeta(t)},~ \rho_2(t):=\mu^{*-3/2}(t)e^{s\zeta^*(t)}, ~ \rho_3(t):=e^{s\widehat \zeta(t)}\widehat\mu^{-3/2}(t),~ \rho_4(t):=\rho_3^{1/2}(t)\quad \forall t\in(0,T).%e^{{s\over 2}\widehat \zeta(t)}\widehat\mu^{-3/4}(t).    
$$
\begin{rmq}
	Notice that $e^{s\widehat\zeta}$ and $e^{s\zeta^*}$ blow up exponentially as $t\to T^-$ and $\widehat\mu$ and $\mu^*$ blow up polynomially as $t\to T^-$.
	
\end{rmq}
\begin{rmq}\label{bounds_weight}
	It is not difficult to deduce the following:
\begin{itemize}
	\item Since $\rho_4^{-1}\in L^\infty(0,T)$, we have that $\rho_4\rho_3^{-1}=\rho_4^{-1}\in L^\infty(0,T)$.
	
	\item If we take $\lambda_0$ large enough, %\footnote{%$m\geq 1+{\ln 2\over \|\eta\|_\infty}$ or 
	 for instance ${\lambda _0} \geq{\ln 2\over \|\eta \|_\infty (m -1)}$, we have that $e^{\lambda m \|\eta \|_\infty} - 2e^{\lambda \|\eta\|_\infty}+ e^{\lambda\eta (1)}>0$.
	%derivada en $\lambda$: $ \|\eta \|_\infty e^{\lambda \|\eta\|_\infty}(m e^{\lambda \|\eta \|_\infty (m -1)} - 2)+ \eta (1)e^{\lambda\eta (1)}$
	Therefore, $\rho_4\rho_2^{-1}\in L^\infty(0,T)$.
	
	\item From $\rho_{4,t}:= e^{s\widehat \zeta/ 2}({s\over2}\widehat\mu^{-3/4}\widehat\zeta_t-{3\over4}\widehat\mu^{-7/4}\widehat\mu_t)$ and by taking $\lambda_0$ large enough, 
	we have that $\rho_{4,t}\rho_0^{-1}\in L^\infty(0,T)$.
	
	%(\rho_{4,t}\rho_0^{-1})\rho_0\widehat z
\end{itemize}
\end{rmq}

	An estimate with such weights is given in the following result. In the proof, we will use Theorem \ref{cl:carlemanadj} and classical energy estimates.
\begin{proposition}\label{lm:caraltweights}
	Under the conditions in~Theorem~\ref{cl:carlemanadj},
%for any $\lambda \ge \lambda _0$, any~$s\ge s_0(T+T^2)$, any~$\vp_T\in H^1(-1,1)$ with $\vp_T(-1)=0$, any~$\gamma_T\in \mathbb{R}$ such that the compatibility condition~\eqref{compatib_cond_2} is satisfied and any source terms $g_1\in L^2(Q)$ and $g_2\in L^2(0,T)$, 
	the unique strong solution to~\eqref{adj:Stefan} satisfies:
\begin{equation}\label{carleman:3}
\begin{alignedat}{2}
    &\!\!\!+\int_0^{T}\left[|\gamma_t|^2+\widehat\mu\left(|\varphi_x(-1,t)|^2+|\varphi_x(1,t)|^2\right)+\widehat\mu^3\left(|\gamma|^2+|\varphi(1,t)|^2\right)\right]e^{-2s\widehat\zeta}\,dt \\
&+\iint_Q\!\! \left[\mu^{-1}(|\vp_t|^2 + |\vp_{xx}|^2)+\mu|\vp_x|^2 
+\mu^3 |\vp|^2\right]e^{-2s\zeta}\,dx\,dt+\|\varphi(\cdot\,,0)\|_{H^1(-1,1)}^2+|\gamma(0)|^2\\  \dis
		&  \le C_2   \left(\jjnt_Q|g_1|^2 e^{-2s \zeta^*}\,dx\,dt
+ \int_0^T |g_2|^2e^{-2s\widehat\zeta}\,dt+\iil  (\mu^*)^{3}|\vp|^2e^{-2s \zeta^*}\,dx\,dt \right),
\end{alignedat}
\end{equation}
	for a positive constant $C_2$ depending on $T$, $s$ and $\lambda$, with $s$ and $\lambda$ as in Theorem \ref{cl:carlemanadj}.
\end{proposition}

\begin{proof}
	It suffices to start from~\eqref{carleman:2} and split the left hand side in two parts, respectively corresponding to the restrictions of~$\varphi$
	to $D_1$ and $D_2$ and the corresponding restrictions of $\gamma$ to~$(0,T/2)$ and~$(T/2,T)$. 

		Let us start by proving the following estimate for system \eqref{adj:Stefan}:
\begin{equation}\label{energy_est}
\begin{alignedat}{2}
		&\|\gamma\|^2_{H^1(0,T/2)}+
		\|\varphi\|^2_{L^2(0,T/2;H^2(-1,1))}+ \|\varphi_t\|^2_{L^2(D_1)}\\
		&\qquad\le~
		e^{C(1+T)}\bigg(\|(g_1,g_2)\|^2_{L^2(0,3T/4;L^2(-1,1))\times L^2(0,3T/4)}\\
		&\quad\qquad+{1\over T^2}\|(\varphi,\gamma)\|^2_{L^2(T/2,3T/4;L^2(-1,1))\times L^2(T/2,3T/4)}\bigg).
\end{alignedat}
\end{equation}
	To do that, let us introduce a function $\kappa\in C^1([0,T])$ with
$$
	\kappa\equiv1\quad\hbox{in}\quad [0,T/2],
	\qquad\kappa\equiv0\quad\hbox{in}\quad [3T/4,T]\quad \hbox{and}
	\quad |\kappa'|\leq C/T,
$$
	for some $C>0$. Using classical energy estimates for the system satisfied by $(\kappa\varphi,\kappa\gamma)$ (see Proposition~\ref{prop:sol_adj}), we obtain:
\begin{align*}
		\|\kappa\gamma\|^2_{ H^{1}(0,T)}+
		\|\kappa\varphi\|^2_{H^{1,2}(Q)}
		\le &~
		e^{C(1+T)}\bigg(\|(\kappa g_1,\kappa g_2)\|^2_{L^2(Q)\times L^2(0,T)}+\|(\kappa'\varphi,\kappa'\gamma)\|^2_{L^2(Q)\times L^2(0,T)}\bigg),
\end{align*}
	which leads to \eqref{energy_est}.
	
	Since the weights are bounded from above and from below, using \eqref{energy_est} we obtain
	a first estimate in $D_1$:
\begin{equation}\label{carleman_001}\begin{alignedat}{2}
&\int_0^{T/2}\left[|\gamma_t|^2+\widehat\mu\left(|\varphi_x(-1,t)|^2+|\varphi_x(1,t)|^2\right)+\widehat\mu^3\left(|\gamma|^2+|\varphi(1,t)|^2\right)\right]e^{-2s\widehat\zeta}\,dt 
\\ \quad \dis
		&\iint_{D_1} \left[\mu^{-1}(|\vp_t|^2 + |\vp_{xx}|^2)+\mu|\vp_x|^2 
+\mu^3 |\vp|^2\right]e^{-2s\zeta}\,dx\,dt +|\gamma(0)|^2+\|\varphi(\cdot\,,0)\|_{H^1(-1,1)}^2\\ \quad \dis
    		& \leq C \left[\int_0^{3T/4}\left(\int_{-1}^1\!|g_1|^2 e^{-2s \zeta}\,dx
+|g_2|^2e^{-2s\widehat\zeta}\right)\,dt\right.\\
&\left.\qquad\qquad+\int_{T/2}^{3T/4}\left(\int_{-1}^{1} \!\! \lambda^4(s\mu)^3|\vp|^2e^{-2s \zeta}\,dx+\lambda^3(s\widehat\mu)^3|\gamma|^2e^{-2s\widehat\zeta}\right)\,dt \right],
		%+s^{-1}\lambda^{-1}\int_0^T\widehat\xi^{-1}e^{-2s\widehat\alpha}|g|^2\,dt\right).    
     \end{alignedat}
\end{equation}
	where $C$ is a positive constant depending on $s$, $\lambda$ and $T$.

	On the other hand, since $\alpha=\zeta$ and $\xi=\mu$ in $D_2$, thanks to Theorem \ref{cl:carlemanadj} we have:
\[
\begin{alignedat}{2}
		&\iint_{D_2}\!\! \left[(s\mu)^{-1}(|\vp_t|^2 + |\vp_{xx}|^2)+\lambda^2(s\mu)|\vp_x|^2 
+\lambda^4(s\mu)^3 |\vp|^2\right]e^{-2s\zeta}\,dx\,dt \\
&\dis+\int_{T/2}^T\left[|\gamma_t|^2+\widehat\mu\left(|\varphi_x(-1,t)|^2+|\varphi_x(1,t)|^2\right)+\widehat\mu^3\left(|\gamma|^2+|\varphi(1,t)|^2\right)\right]e^{-2s\widehat\zeta}\,dt   \\ \quad \dis
%		=&\int_{T\over2}^{T}\!\!\!\int_{-1}^1\!\! \left[(s\xi)^{-1}(|\vp_t|^2 + |\vp_{xx}|^2)+\lambda^2(s\xi)|\vp_x|^2 
%+\lambda^4(s\xi)^3 |\vp|^2\right]e^{-2s\alpha}+\int_{T\over2}^{T}\!\!\!\left( |\gamma_t|^2+\lambda^3(s\widehat\xi)^3|\gamma|^2\right)e^{-2s\widehat\alpha}\\
%&\dis+\int_{T\over2}^{T}\lambda^3(s\widehat\xi)^3 |\varphi(1,t)|^2e^{-2s\widehat\alpha}+\int_{T\over2}^{T}\lambda(s\widehat\xi)(|\varphi_x(-1,t)|^2
%		+|\varphi_x(1,t)|^2)e^{-2s\widehat\alpha}  \\ \quad \dis
		\leq&\iint_Q\!\! \left[(s\xi)^{-1}(|\vp_t|^2 + |\vp_{xx}|^2)+\lambda^2(s\xi)|\vp_x|^2 
+\lambda^4(s\xi)^3 |\vp|^2\right]e^{-2s\alpha}\,dx\,dt\\
&\dis+\int_{0}^{T}\!\!\!\left[ |\gamma_t|^2+\lambda(s\widehat\xi)\left(|\varphi_x(-1,t)|^2
		+|\varphi_x(1,t)|^2\right)+\lambda^3(s\widehat\xi)^3\left(|\gamma|^2+ |\varphi(1,t)|^2\right)\right]e^{-2s\widehat\alpha}\,dt \\ \quad \dis
		\leq&~ C_0 \left(\jjnt_Q|g_1|^2 e^{-2s \alpha}
+ \int_0^T |g_2|^2e^{-2s\widehat\alpha}+s^3\lambda^4\iil  \xi^3|\vp|^2e^{-2s \alpha}\right).
\end{alignedat}
\]
	
	Finally, from the definition of $\zeta$, $\mu$ and $\widehat\zeta$, we deduce that
\begin{align*}
		&\iint_{D_2}\!\!\! \left[(s\mu)^{-1}(|\vp_t|^2 + |\vp_{xx}|^2)+\lambda^2(s\mu)|\vp_x|^2 
+\lambda^4(s\mu)^3 |\vp|^2\right]e^{-2s\zeta}\,dx\,dt\\
&\dis+\int_{T/2}^T\left[|\gamma_t|^2+\widehat\mu\left(|\varphi_x(-1,t)|^2+|\varphi_x(1,t)|^2\right)+\widehat\mu^3\left(|\gamma|^2+|\varphi(1,t)|^2\right)\right]e^{-2s\widehat\zeta}\,dt   \\ \quad \dis
		 \le &~C(T,s,\lambda)\left(\jjnt_Q|g_1|^2 e^{-2s \zeta}\,dx\,dt
+ \int_0^T |g_2|^2e^{-2s\widehat\zeta}\,dt+\iil  \mu^3|\vp|^2e^{-2s \zeta}\,dx\,dt \right),
\end{align*}
	which, combined with \eqref{carleman_001}, provides \eqref{carleman:3}.	
\end{proof}

	In the sequel, we will use the notation
\begin{equation*}
C_\rho^k([0,T];B):= \{v: \rho v\in C^k([0,T];B)\}~\hbox{and} ~W^{r,p}_\rho(0,T;B):= \{v: \rho v\in W^{k,r}(0,T;B)\}.
\end{equation*}
	Here, it is assumed that~$B$ is a Banach space, $\rho:[0,T]\mapsto \mathbb{R}$ is a positive measurable function, $k\in \mathbb{N}$, $r\in \mathbb{R}_{\geq0}$ and~$p \in [1,+\infty]$.
	Accordingly, we set
\[
	\|v\|_{C^k_\rho([0,T];B)}:= \|\rho v\|_{C^k([0,T];B)}\ \hbox{ and } \  \|v\|_{W^{r,p}_\rho(0,T;B)}:= \|\rho v\|_{W^{r,p}(0,T;B)}.
\]
	In particular, when~$B = \R$, we simply write $C_\rho^k([0,T])$ and~$W^{r,p}_\rho(0,T)$; when $p=2$, we use the notation $H^r(0,T;B):=W^{r,2}(0,T;B)$ and $H^r(0,T):=W^{r,2}(0,T)$.

	We will also need the spaces $H^{1,2}_\rho(Q):=\{v: \rho v\in H^{1,2}(Q)\}$ and $H^{1,2}_{0,\rho}(Q):=\{v: \rho v\in H^{1,2}_0(Q)\}$, endowed with the norm
	$\|v\|_{H^{1,2}_\rho(Q)}:=\|\rho v\|_{H^{1,2}(Q)}$.

	Let us establish the null controllability of~\eqref{eq:conStefantransnulllinear} with a right hand side which decays exponentially as $t\to T^-$. 
	As we will see in the next section, this will be crucial to deduce the local null controllability of \eqref{eq:conStefantransnull_extended}.

	Let us introduce the linear operators
$$
	\cal L_1(z,h) :=\bar q z_t-z_{xx}+\frac{x}{\beta}\bar p_x(1,\cdot) z_x+\frac{x}{\beta}\bar p_x z_x(1\,,\cdot)+{\black 2\over \beta}\bar p_th 
	\quad \hbox{and} \quad
	\cal L_2(z,h) :=h_t+z_x(1,\cdot)
$$
	and the space $E$, given by
$$
\begin{alignedat}{2}
	E:=&\,\big\{(z,h,w)\in L^2_{\rho_0}(Q)\times L^2_{\rho_1}(0,T)\times L^2_{\rho_2}(\omega\times(0,T)): \cal L_1(z,h)-w1_\omega\in L_{\rho_3}^2(Q),~\cal L_2(z,h)\in L_{\rho_3}^2(0,T)\\
	&\qquad\qquad\qquad\,h\in H^1_{\rho_4}(0,T)~ \hbox{and}~ z\in H^{1,2}_{0,\rho_4}(Q) \big\}.
	\end{alignedat}
$$
	It is clear that $E$ is a Hilbert space for the norm $\|\cdot\|_E$, where
$$
\begin{alignedat}{2}
	\|(z,h,w)\|_E:=&~ \bigg(\|(z,h,w1_\omega)\|_{L^2_{\rho_0}(Q)\times L^2_{\rho_1}(0,T)\times L^2_{\rho_2}(Q)}^2+\|\cal L_1(z,h)-w1_\omega\|_{L_{\rho_3}^2(Q)}^2\\
	&+
	\|\cal L_2(z,h)\|_{L_{\rho_3}^2(0,T)}^2+\|h\|_{H^1_{\rho_4}(0,T)}^2+\|\rho_4 z\|_{H^{1,2}(Q)}^2\bigg)^{1/2}.
		\end{alignedat}
$$

	The null controllability of the linearized system is guaranteed by the following result:
\begin{proposition}\label{prop:contlin}
	Assume that %~$e^{s\widehat \zeta}\widehat\mu^{-3/2}f_1\in L^2(Q)$, $e^{s\widehat \zeta}\widehat\mu^{-3/2}f_2\in L^2(0,T)$
	$(f_1,f_2)\in L^2_{\rho_3}(Q)\times L^2_{\rho_3}(0,T)$
	 and that ~$(z_0,h_0)\in H_0^1(-1,1)\times \mathbb{R}$.
	Then, there exists a solution to~\eqref{eq:conStefantransnulllinear} satisfying $(z,h)\in E$.
\end{proposition}

\begin{proof}
	Let us consider the following subspace of $H^{1,2}(Q)\times H^1(0,T)$:
\[
P_0=\!\{\!(\varphi,\gamma)\!\in\! H^{1,2}(Q)\!\times\! H^1(0,T): \varphi(\cdot\,,-1)\!=0, ~\vp(1,\cdot)-\gamma-\frac{1}{\beta}\int_{-1}^1\! \!\!\bar p_x(x,\cdot) x \vp(x,\cdot)\,dx=0~\hbox{in}~(0,T)\}\!. 
\]
	Let ${\mathcal A}: P_0\times P_0\mapsto \bb R$ be the bilinear form
\[
\begin{alignedat}{2}
	{\mathcal A}((\widehat\varphi,\widehat\gamma),(\varphi,\gamma)) 
	:=& \iil \rho_2^{-2}\widehat\varphi\varphi  \,dx\,dt
	+ \jjnt_Q \rho_0^{-2}{\cal L}^*_1(\widehat\varphi,\widehat\gamma) \cal L^*_1(\varphi,\gamma) \,dx\,dt+ \int_0^T \rho_1^{-2} {\cal L}_2^*(\widehat\varphi,\widehat\gamma) \cal L^*_2(\varphi,\gamma) \,dt
	\end{alignedat}
\]
	and let ${\mathcal F}:P_0\mapsto\bb R$ be the linear form
\begin{equation*}
{\mathcal F}(\varphi,\gamma) := \ol{q}(0)\int_0^1 z_0(x)\cdot\varphi(x,0)\,dx +\beta h_0\gamma(0)
+ \jjnt_Q f_1\varphi \,dx\,dt + \int_0^T f_2\gamma\,dt,
\end{equation*}
where
$$
	\cal L^*_1(\phi,\gamma) :=-\bar q\vp_t- \vp_{xx}-\frac{x}{\beta}\bar p_x(1,\cdot)\vp_x+\frac{1}{\beta}\bar p_x(1,\cdot)\vp \quad \hbox{and}\quad 
	\cal L^*_2(\phi,\gamma) :=\gamma_t-\int_{-1}^1{\black 2\over \beta}\ol p_t(x,\cdot)\vp(x,\cdot)\,dx.
$$

% An easy consequence is that
%\begin{equation}\label{comp:alphabeta}
%\beta^*\leq\alpha^* \mbox{ in } [0,T],
%\end{equation}
%where 
%$$
%	\alpha^*(t):= \min\limits_{x\in [-1,1]}\alpha(x,t).
%$$

	Note that the observability inequality \eqref{carleman:3} holds for every $(\phi,\kappa)\in P_0$.
	Consequently, ${\mathcal A}(\cdot,\cdot)$ is a scalar product in $P_0$ and there exists $C>0$ such that, for all $(\varphi,\gamma)\in P_0$, the following estimate holds:
	\[
|{\mathcal F}(\varphi,\gamma)|\leq 
C\left(\|z_0\|_{L^2(-1,1)}+|h_0|+ 
\|f_1\|_{L^2_{\rho_3}(Q)}+ \|f_2\|_{L^2_{\rho_3}(0,T)}
\right) \sqrt{{\mathcal A}((\varphi,\gamma),(\varphi,\gamma))}.
	\]

	In the sequel, we will denote by~$P$ the completion of~$P_0$ for the scalar product ${\mathcal A}$. We will still denote by~${\mathcal A}$ and~${\mathcal F}$ the corresponding continuous extensions. Note that $P$ can be identified with the Hilbert space
\[
\begin{alignedat}{2}
	&\{(\varphi,\gamma)\in L^2_{loc}(Q_T)\times L^2_{loc}(0,T) : {\mathcal A}((\varphi,\gamma),(\varphi,\gamma))< +\infty, ~\,\varphi|_{\{-1\}\times(0,T)}=0,\\
	&\qquad\qquad~\vp(1,\cdot)-\gamma-\frac{1}{\beta}\int_{-1}^1\! \!\!\bar p_x(x,\cdot) x \vp(x,\cdot)\,dx=0~\hbox{in}~(0,T)
	~\hbox{and $(\varphi,\gamma)$ satisfies  \eqref{carleman:3}}\}.
	\end{alignedat}
\]
	
	From Lax-Milgram Theorem, there exists a unique $(\widehat\varphi,\widehat\gamma)\in P$ such that
\begin{equation}\label{def:zgammav}
{\mathcal A}((\widehat\varphi,\widehat\gamma),(\varphi,\gamma))={\mathcal F}(\varphi,\gamma) \ \ \forall (\varphi,\gamma) \in P.
\end{equation}
	Let us introduce~$(\widehat z,\widehat h,\widehat w)$, with
\[
(\widehat z,\widehat h):= (\rho_0^{-2} {\cal L}^*_1(\widehat\varphi,\widehat\gamma),\rho_1^{-2}  {\cal L}_2^*(\widehat\varphi,\widehat\gamma)),\ \ 
 \widehat w= -\rho_2^{-2} \widehat\varphi1_{\omega}.
\]

	From~\eqref{def:zgammav}, we get:
\[
 \jjnt_Q \rho_0^2|\widehat z|^2 \,dx\,dt + \int_0^T \rho_1^2|\widehat h|^2 \,dt+\iil\rho_2^2|\widehat w|^2\,dx\,dt ={\mathcal A}((\widehat\varphi,\widehat\gamma),(\widehat\varphi,\widehat\gamma)) ={\mathcal F}(\widehat\varphi,\widehat\gamma).
\]
	Therefore, taking into account the continuity of ${\mathcal F}$, we have:
\begin{equation}\label{eq:weighted_hat}
\begin{alignedat}{2}
\!\!\jjnt_Q\!\!\rho_0^2|\widehat z|^2 \,dx\,dt + \int_0^T\!\!\!\rho_1^2|\widehat h|^2 \,dt+\iil\!\rho_2^2|\widehat w|^2\,dx\,dt \leq C\left(\|z_0\|_{L^2(-1,1)}^2+|h_0|^2+ 
\|f_1\|^2_{L^2_{\rho_3}(Q)}+ \|f_2\|^2_{L^2_{\rho_3}(0,T)}
\right).
\end{alignedat}
\end{equation} 

Note that, in particular, $(\widehat z,\widehat h,\widehat w)\in L^2(Q)\times L^2(0,T)\times L^2(\omega\times (0,T))$. Then, 
from \eqref{def:zgammav}, we see that $(\widehat z,\widehat h)$ is the unique solution by transposition of~\eqref{eq:conStefantransnulllinear} with
$w=\widehat w$, see~Proposition \ref{ex_un_mixed_sol}. Thanks to the fact that the $z_0$, $\widehat w$, $f_1$ and $f_2$ are sufficient regular, 
Proposition \ref{prop:sol_adj_7} guarantees that $(\widehat z,\widehat h)$ is indeed the strong solution of \eqref{eq:conStefantransnulllinear} in $H^{1,2}_{0}(Q)\times H^1(0,T)$.	
	
	Let us finally prove that $(\widehat z,\widehat h,\widehat w)\in E$.
	
	Using ~\eqref{eq:conStefantransnulllinear}  and \eqref{eq:weighted_hat}, we can easily check that $\widehat z\in L^2_{\rho_0}(Q)$,
	$\widehat h\in L^2_{\rho_1}(0,T)$,
	$\widehat w\in L^2_{\rho_2}(\omega\times(0,T))$,
	$\cal L_1(\widehat z,\widehat h)-\widehat w1_\omega\in L_{\rho_3}^2(Q)$ and $\cal L_2(\widehat z,\widehat h)\in L_{\rho_3}^2(0,T)$.
	
	It remains to check that $\widehat h\in H^1_{\rho_4}(0,T)$ and $\widehat z\in H^{1,2}_{0,\rho_4}(Q)$. With that purpose, we define $\widetilde z=\rho_4\widehat z$ and $\widetilde h=\rho_4 \widehat h $.
	Then, $(\widetilde z, \widetilde h)$ is the solution to the system:
\begin{equation}\label{eq:conStefantransnulllinear22}
\left\{
	\begin{array}{lcl}\dis
		%\bar q z_t- z_{xx}-\frac{1}{2}\bar q' x z_x+ \bar p_t h+\frac{1}{\beta}\bar p_x x z_x(1\,,\cdot)
		\cal L_1(\widetilde z,\widetilde h)=(\rho_4\rho_3^{-1})\rho_3f_1+(\rho_4\rho_2^{-1})\rho_2w1_\omega+(\rho_{4,t}\rho_0^{-1})\rho_0\widehat z	& \mbox{ in }&Q,\\
 		\widetilde z(-1,\cdot)= 0																& \mbox{ in }&(0,T),\\
		\widetilde z(1,\cdot)=0																	& \mbox{ in }&(0,T),\\
		\widetilde z(\cdot\,,0)= \rho_4(0)z_0																& \mbox{ in }&(-1,1),\\
		\cal L_2(\widetilde z,\widetilde h)=\rho_4f_2+\rho_{4,t}h													& \mbox{ in }& (0,T),\\
		\widetilde h(0)= \rho_4(0)h_0.&&
	\end{array}
	\right.
\end{equation}

	Consequently, thanks to Remark \ref{bounds_weight} and Proposition~\ref{prop:sol_adj_7}, we obtain the desired estimates and ~$(z,h,w)\in E$, as desired.
\end{proof}

%%%%%%%%%%%%%%%%%%%%%%%%%%%%%%%%%%%%%%%%%%
%%%%%%%%%%%%%%%%%%%%%%%%%%%%%%%%%%%%%%%%%%
%%%%%%%%%%%%%%%%%%%%%%%%%%%%%%%%%%%%%%%%%%
%% REFORMULATION OF THE free-boundary PROBLEM   %%%%%%%%%
%%%%%%%%%%%%%%%%%%%%%%%%%%%%%%%%%%%%%%%%%%
%%%%%%%%%%%%%%%%%%%%%%%%%%%%%%%%%%%%%%%%%%
%%%%%%%%%%%%%%%%%%%%%%%%%%%%%%%%%%%%%%%%%%

\subsection{Controllability of the nonlinear system}\label{sec:controlnonlinear}

	We now prove the controllability of~\eqref{eq:conStefantransnull_extended} by applying a local inversion theorem.
	More precisely, we are going to use the following result, whose proof can be found for instance in~\cite[Chapter~2, p.~107]{alekseevoptimal}:

\begin{theorem}[Liusternik-Graves' Theorem]\label{tm:invfunc}
	Let $\cal B_1$ and $\cal B_2$ be two Banach spaces and let $\Lambda: \cal B_1 \mapsto \cal B_2$ be of class $C^1$ in a neighborhood of~$b_{1,0} \in\cal B_1$.
	Assume that $\Lambda(b_{1,0}) = b_{2,0}$ and~$\Lambda'(b_{1,0}):\cal B_1\mapsto\cal B_2$ is surjective.
	Then, there exists $\delta > 0$ such that, for every $b_2 \in \cal B_2$ satisfying $\| b_2 - b_{2,0} \|_{\cal B_2} \leq \delta$, there exists at least one solution~$b_1 \in \cal B_1$ to the equation $\Lambda(b_1) = b_2$.
\end{theorem}

	We shall apply this result with $\cal B_1=E$, $\cal B_2 = F_1\times F_2$ and 
\begin{equation}\label{def:calG}
	 \Lambda(z,h,w) = \left(\cal L_1 (z,h)-w1_\omega+{\black 2\over \beta}hz_t+\frac{x}{\beta}z_x(1,\cdot)z_x,\, \cal L_2 (z,h),\,z(\cdot\,,0),\,h(0)\right)
\end{equation}
	for every $(z,h,w) \in E$.
	Here, we have introduced the Hilbert spaces $F_1:=L^2_{\rho_3}(Q)\times L^2_{\rho_3}(0,T)$ for the right hand sides and~$F_2:=H_0^1(-1,1)\times \mathbb{R}$ for the initial conditions.

	Since $\Lambda$ contains linear and bilinear terms and thanks to the definition of~$E$, it is not difficult to check that $\Lambda$ is continuous.
	Indeed, we only have to prove that the bilinear form
	$$
	 ((z_1,h_1,w_2),(z_2,h_2,w_2))\to {\black 2\over \beta}h_1z_{2,t}+\frac{x}{\beta}z_{1,x}(1,\cdot)z_{2,x}
	$$
		is bounded from $E\times E$ to $L^2_{\rho_3}(Q)$. This is true because $h_1\in H^1_{\rho_4}(0,T)$ and $z_1,z_2\in H^{1,2}_{0,\rho_4}(Q)$ and, in particular, we have 
		$\rho_4 h_1\in H^1(0,T)$, $\rho_4 z_{2,t}\in L^2(Q)$, $z_{1,x}(1,\cdot)\in L^2(0,T)$ and $\rho_4 z_2\in C^0([0,T];H_0^1(-1,1))$.

	Therefore, $\Lambda\in C^1(\cal B_1;\cal B_2)$.
	
	On the other hand, note that $ \Lambda'(0,0,0): \cal B_1 \mapsto \cal B_2$ is given by
\[
	\Lambda'(0,0,0)(z,h,v) = (\cal L_1 (z,h,w),\cal L_2 (z,h,w), z(\cdot\,,0),h(0)) \quad \forall (z,h,v)\in \cal B_1 .
\]
	In view of the null controllability result for \eqref{eq:conStefantransnulllinear} given in Proposition~\ref{prop:contlin},  $\Lambda'(0,0,0)$ is surjective.

	Consequently, we can apply Theorem~\ref{tm:invfunc} with these data and the proof of Proposition \ref{onephase:fixeddomaincontrol_exdended} is achieved.
 
\paragraph{Acknowledgements.}
	EFC and DAS were partially supported by  Grant~PID$2020$--$114976$GB--I$00$, funded by MCIN/AEI/$10.13039/501100011033$.
	DAS was partially supported by  Grant IJC$2018$--$037863$-I funded by MCIN/AEI/$10.13039/501100011033$.

%%%%%%%%%%%%%%%%%%%

\begin{appendices}

\section{Proof of Proposition \ref{prop:sol_adj}}\label{sec:app_A}

	Recall that $(\cdot\,,\cdot)_2$ and~$\|\cdot\|_2$ stand for the usual scalar product and norm in~$L^2(-1,1)$.
	On the other hand, we will denote by~$\|\cdot\|_\infty$ the usual norm in~${L^\infty(Q)}$. 

	The proof of existence relies on {\it Leray-Schauder's Fixed-Point Principle} (see for instance~\cite{Zeidler}). For convenience, let us recall this important result: 
	
\begin{theorem}\label{leray-Schauder}
	Let $\mathcal{B}$ be a Banach space and let $\Lambda:\mathcal{B}\times[0,1]\mapsto \mathcal{B}$ be a continuous and compact mapping such that
\begin{itemize}
	\item $\Lambda(x,0)=0$ for all $x\in \mathcal{B}$.
	\item There exists $M>0$ such that, for any pair $(x,\sigma)\in \mathcal{B}\times [0,1]$ satisfying $x=\Lambda (x,\sigma)$, one has
$$
	\|x\|_{\mathcal{B}} \leq M.
$$
	Then, there exists a least one fixed-point of the mapping $\Lambda_1: \mathcal{B}\mapsto \mathcal{B}$, given by
$$	
	\Lambda_1(x)=\Lambda(x,1) \quad \forall x\in \mathcal{B}.
$$
\end{itemize}	
\end{theorem} 

	%In order to prove the existence of a strong solution, % we are going to use a fixed point argument. 
%For that, it is necessary to assume the compatibility condition:
%\begin{equation}\label{eq:compphigamma}
%\psi^0(1)= 2\gamma^0+\int_0^1N(x,0)\psi^0(x)\,dx.    
%\end{equation}
%
%In order to do that, 
Let us consider the mapping~$\Lambda: \mathcal{B} \times [0,1] \mapsto \mathcal{B}$, given by 
$\Lambda((\widehat\varphi,\widehat\gamma),\sigma)= (\varphi,\gamma)$, where 
	$$
\mathcal{B}= \{ (\widehat\varphi,\widehat\gamma)\in H^{3/4}(0,T;L^2(-1,1))\times H^{3/4}(0,T): \widehat\varphi(\cdot\,,T)=\varphi_T , \ \ \widehat\gamma(T)=\gamma_T \}
	$$
 and $(\varphi,\gamma)$ is the unique solution to
\begin{equation}\label{adj:Stefanplainfixed}
\left\{
	\begin{array}{lll}
	-\bar q(t) \varphi_t-\varphi_{xx}-a\varphi_x-b\varphi=\sigma f& \mbox{ in }&Q,\\
	\noalign{\smallskip}\displaystyle
	\varphi(-1,\cdot)=0& \mbox{ in }&(0,T),\\
	\noalign{\smallskip}\displaystyle
	\varphi(1,t)= \sigma\left(\widehat\gamma(t)+(N(\cdot\,,t),\widehat\varphi(\cdot\,,t))_2\right)& \mbox{ in }&(0,T),\\
	\noalign{\smallskip}\displaystyle
	\varphi(T)=\sigma\varphi_T& \mbox{ in }&(-1,1),\\
	\noalign{\smallskip}\displaystyle
	\gamma'(t)=(R(\cdot\,,t),\varphi(\cdot\,,t))_2+\sigma g(t)& \mbox{ in }&(0,T),\\
	\noalign{\smallskip}\displaystyle
	\gamma(T)=\sigma\gamma_T.
\end{array}
\right.
\end{equation}
% It is trivial to prove that $\cal X_T$ is a close space. 
%Moreover, $\Lambda(\cal X_T)\subset \cal X_T$, which can be proved using standard regularity results for the heat equation  as:
%\[\varphi^0(1)-2\hat\gamma(0)-\int_0^1N(x,0)\hat\varphi(x,0)\,dx=
%\varphi^0(1)-2\gamma^0-\int_0^1N(x,0)\varphi^0(x)\,dx=0.
%\]
\begin{rmq}
	Since $H^{s}(0,T)\hookrightarrow C^{0,s-1/2}([0,T])$ for any $s\in(1/2,1)$, we see that the initial conditions in the definition of $\mathcal{B}$ make sense.
\Fin
\end{rmq}

\begin{rmq}\label{rmq_estimate_fixedpoint}
	The unique solution $(\varphi,\gamma)$ to~\eqref{adj:Stefanplainfixed} belongs to~$H^{1,2}(Q)\times H^1(0,T)$ and depends continuously in this space with respect to $\widehat\varphi$ and $\widehat \gamma$.
	Indeed, 
%it is enough to verify that the non-homogeneous boundary condition belongs to~$H^{3/4}(0,T)$ and satisfies a compatibility condition for $(\widehat \varphi,\widehat \gamma) \in \mathcal{B}$:
	note that $\widehat \gamma\in H^{3/4}(0,T)$ by definition and $t\mapsto (N(\cdot\,,t),\widehat \varphi(\cdot\,,t))_2$ belongs to $H^{3/4}(0,T)$, since $H^{3/4}(0,T)$ is a Banach algebra (see~for instance~\cite{di2012hitchhiker}, though the results date back to \cite{strichartz1967multipliers}).
	Thanks to the compatibility condition \eqref{compatib_cond} and the equation satisfied by $\varphi$, we have
	$$
\sigma\varphi_T(1)=\varphi(1,T)=\sigma\left[\gamma_T+(N(\cdot\,,T),\varphi_T)_2\right].
	$$
	This implies that $(\varphi,\gamma)\in H^{1,2}(Q)\times H^1(0,T)$ and, moreover, there exists a constant $C>0$ such that
	\begin{equation}\label{ineq_fixedpoint}
\begin{alignedat}{2}
	\|(\varphi,\gamma)\|_{ H^{1,2}(Q)\times H^1(0,T)}^2
\leq \sigma^2 e^{C(1+T)}&\left(
\|( f,g)\|_{L^2(Q)\times L^2(0,T)}^2+\|(\varphi_T,\gamma_T)\|_{H^1(-1,1)\times \mathbb{R}}^2\right.\\
&\left.+\|(\widehat \varphi,\widehat \gamma)\|^2_{H^{3/4}(0,T;L^2(-1,1))\times H^{3/4}(0,T)}\right),
\end{alignedat}
	\end{equation}
whence the continuous dependence is ensured.
\Fin
\end{rmq}

	As in \cite{brandao_cara}, we are going to prove that $\Lambda$ fulfills the assumptions in Theorem~\ref{leray-Schauder}:

%\noindent
$\bullet$ $\Lambda: \mathcal{B} \times [0,1] \mapsto  \mathcal{B}$ is well-defined and continuous.
	Indeed, this follows from Remark \ref{rmq_estimate_fixedpoint} and the fact that $H^{1,2}(Q)\times H^1(0,T)\hookrightarrow \mathcal{B}$;
%\item  $\Lambda: \mathcal{B} \times [0,1] \mapsto  \mathcal{B}$ is continuous. {\black\bf(few words to justify that)}
	
%\noindent
$\bullet$ $\Lambda: \mathcal{B} \times [0,1] \mapsto  \mathcal{B}$ is compact as a consequence of parabolic regularity. 	Indeed, whenever $((\widehat\varphi,\widehat\gamma),\sigma)$ belongs to a bounded set in the space $ \mathcal{B} \times [0,1]$, we see from Remark~\ref{rmq_estimate_fixedpoint} that the associated $(\varphi,\gamma)$ belongs to a bounded set in $H^{1,2}(Q)\times H^1(0,T)$.
	But this space is compactly embedded in $\mathcal{B}$:
	note that $H^{s_2}(0,T)$ is compactly embedded in $H^{s_1}(0,T)$ for any $s_1<s_2$, $H^{1,2}(Q)$ is continuously embedded in $H^{r}(0,T;H^{2(1-r)}(-1,1))$ for any $r\in[0,1]$ and, also, $H^{r}(0,T;H^{2(1-r)}(-1,1))$  is compactly embedded in $H^{3/4}(0,T;L^2(-1,1))$ for any $r\in(3/4,1)$.
	We deduce that $\Lambda$ is compact.

%\noindent
$\bullet$ Obviously, $\Lambda((\widehat\varphi,\widehat\gamma),0)= (0,0)$ for all $(\widehat\varphi,\widehat\gamma)\in \mathcal{B}$.

%\noindent
$\bullet$ There exists $C>0$ (depending on $R$, $N$ and $\bar q$, but independent of $\sigma$) such that, for any~$((\varphi,\gamma),\sigma)\in \mathcal{B}\times [0,1]$ satisfying $(\varphi,\gamma)=\Lambda ((\varphi,\gamma),\sigma)$, one has
	\begin{equation}\label{ineq_fixedpointhypo}
\begin{alignedat}{2}
\|(\varphi,\gamma)\|_{ H^{1,2}(Q)\times H^1(0,T)}^2
\leq e^{C(1+T)}&\left(
\|( f,g)\|_{L^2(Q)\times L^2(0,T)}^2+\|(\varphi_T,\gamma_T)\|_{H^1(-1,1)\times \mathbb{R}}^2\right).
\end{alignedat}
	\end{equation}	
	Let us prove this.
	Let $(\varphi,\gamma)\in H^{1,2}(Q)\times H^1(0,T)$ be a solution to
\begin{equation}\label{adj:Stefanplainfixed2}
\left\{
	\begin{array}{lll}
	-\bar q(t)\varphi_t-\varphi_{xx}-a\varphi_x-b\varphi=\sigma f& \mbox{ in }&Q,\\
	\noalign{\smallskip}\displaystyle
	\varphi(-1,\cdot)=0& \mbox{ in }&(0,T),\\
	\noalign{\smallskip}\displaystyle
	\varphi(1,t)= \sigma\left[\gamma(t)+(N(\cdot\,,t),\varphi(\cdot\,,t))_2\right]& \mbox{ in }&(0,T),\\
	\noalign{\smallskip}\displaystyle
	\varphi(\cdot\,,T)=\sigma\varphi_T& \mbox{ in }&(-1,1),\\
	\noalign{\smallskip}\displaystyle
	\gamma'(t)=(R(\cdot\,,t),\varphi(\cdot\,,t))_2+\sigma g(t)& \mbox{ in }&(0,T),\\
	\noalign{\smallskip}\displaystyle
	\gamma(T)=\sigma\gamma_T.
\end{array}
\right.
\end{equation}
	Let us multiply $\eqref{adj:Stefanplainfixed2}_1$ successively  by $\varphi$ and $\varphi_t$ and let us integrate in $[\tau,T]\times[-1,1]$, with $\tau\in[0,T)$.
	Then, from Young and Cauchy-Schwarz inequalities, we obtain: 
\begin{equation}\label{eq:adjphi}
\begin{alignedat}{2}
  	{\bar q(\tau)}\|\varphi(\cdot\,,\tau)\|^2_{2}+\int_\tau^T\|\varphi_x(\cdot\,,t)\|^2_{2}\,dt
	\leq&~\sigma^2{\bar q(T)}\|\varphi_T\|^2_{2}+\sigma^2 \|f\|_2+2\int_\tau^T\varphi(1,t)\varphi_x(1,t)\,dt\\
	&+\int_\tau^T \left(2+\|a(\cdot\,,t)\|^2_\infty +\|b(\cdot\,,t)\|^2_\infty\right)\|\varphi(\cdot\,,t)\|^2_{2}\,dt
\end{alignedat}    
\end{equation}
	 and
\begin{equation}\label{eq:adjphitt}
\begin{split}
	\|\varphi_x(\cdot\,,\tau)\|^2_{2}+\int_\tau^T \bar q(t)\|\varphi_t(\cdot\,,t)\|^2_{2}\,dt
	\leq&~ 2\int_\tau^T\varphi_x(1,t)\varphi_{t}(1,t)\,dt+\sigma^2\|\varphi_{T,x}\|^2_{2}+{2\sigma^2\over q_*}\|f\|^2_2
	\\& + {4\over q_*}\int_\tau^T\left(\|a(\cdot\,,t)\|^2_\infty+\|b(\cdot\,,t)\|^2_\infty\right)\|\varphi(\cdot\,,t)\|^2_{H^1(-1,1)}\,dt.
	\end{split}
\end{equation}
	Consequently, combining \eqref{adj:Stefanplainfixed2}$_1$, \eqref{eq:adjphi} and~\eqref{eq:adjphitt} and using that the $H^1$~norm can be interpolated by the $L^2$ and~$H^2$ norms and the fact that $\sigma\in[0,1]$, we deduce that
	\begin{equation}\label{est:regaux1}
\begin{alignedat}{2}
	&\int_\tau^T\left(\|\varphi_t(\cdot\,,t)\|^2_2+\|\varphi(\cdot\,,t)\|^2_{H^2(-1,1)}\right)\,dt + \|\varphi(\cdot\,,\tau)\|_{H^1(-1,1)}^2\\
	 &\leq C\bigg(\|\varphi_T\|_{H^1(-1,1)}^2 + \|f\|^2_2+\int_\tau^T\left(1+\|a(\cdot\,,t)\|^2_\infty +\|b(\cdot\,,t)\|^2_\infty\right)\|\varphi(\cdot\,,t)\|^2_{H^1(-1,1)}\,dt\\
	 &\quad+\int_\tau^T \left(|\varphi(1,t)| + |\varphi_t(1,t)|\right)|\varphi_x(1,t)|\,dt\bigg).
\end{alignedat}
	\end{equation}
	In order to conclude, we have to estimate the boundary terms.
	The first one can be easily bounded with the help of trace interpolation.
	Indeed, recall that $|\phi(1)|\leq C_{0,s}\|\phi\|_{H^{s}(-1,1)}$ for any $\phi\in H^{s}(-1,1)$, $|\phi_x(1)|\leq C_{1,s}\|\phi\|_{H^{s+1}(-1,1)}$ for any $\phi\in H^{s+1}(-1,1)$ with~$s>{1\over2}$ and, also, that the following interpolation inequality holds: $\|\phi\|_{H^r(-1,1)}\leq C_r \|\phi\|_{L^2(-1,1)}^{1-r/2} \|\phi\|^{r/2}_{H^2(-1,1)}$ for all $r\in[0,2]$.
	This gives
\begin{equation}\label{est:regaux2}
	\!\int_\tau^T|\varphi(1,t)||\varphi_x(1,t)|\,dt\leq
\delta \int_\tau^T \|\varphi(\cdot\,,t)\|^2_{H^2(-1,1)}\,dt
+ C_\delta \int_\tau^T\|\varphi(\cdot\,,t)\|^2_{2}\,dt
\end{equation}
for any $\delta>0$.
	For the second boundary term we can use \eqref{adj:Stefanplainfixed2}$_3$ and obtain:
\begin{equation}\label{est:regaux3}
\begin{alignedat}{2}
	\!\!\!\!\int_\tau^T|\varphi_x(1,t)||\varphi_{t}(1,t)|\,dt=&\int_\tau^T|\varphi_x(1,t)|\left|\sigma\left[(R(\cdot\,,t),\varphi(\cdot\,,t))_2+\sigma g(t)+{d\over dt}(N(\cdot\,,t),\varphi(\cdot\,,t))_2\right]\right|\,dt\\
	\leq& \int_\tau^T|\varphi_x(1,t)|\big|\left[(R(\cdot\,,t),\varphi(\cdot\,,t))_{2}+\sigma g(t)+(N_t(\cdot\,,t),\varphi(\cdot\,,t))_2\right]\big|\,dt \\
	&+\int_\tau^T|\varphi_x(1,t)|\big|(N(\cdot\,,t),\varphi_t(\cdot\,,t))_2\big|\,dt\\
	 \leq& ~\delta \int_\tau^T \|\varphi_{t}(\cdot\,,t)\|^2_{2}\,dt+\|g\|^2_{L^2(0,T)}+C_\delta\int_\tau^T \left(1+\|N(\cdot\,,t)\|_2^2\right)|\varphi_x(1,t)|^2\,dt \\
	 &+\int_\tau^T(\|R(\cdot\,,t)\|_2^2+\|N_t(\cdot\,,t)\|_2^2)\|\varphi(\cdot\,,t)\|^2_{2}\,dt\\
	\leq &~ \delta \int_\tau^T \left(\|\varphi_{t}(\cdot\,,t)\|^2_{2}+\|\varphi\|^2_{H^2(-1,1)}\right)\,dt +\|g\|^2_{L^2(0,T)}
	\\
	&+C_\delta\int_\tau^T\left(1+\|N(\cdot\,,t)\|_2^{16}+\|R(\cdot\,,t)\|_2^2+\|N_t(\cdot\,,t)\|_2^2\right)\|\varphi(\cdot\,,t)\|^2_{2}\,dt.
\end{alignedat}
\end{equation}
	Finally, combining $\eqref{adj:Stefanplain}_5$, \eqref{est:regaux1}, \eqref{est:regaux2} and~\eqref{est:regaux3} and using Gronwall's inequality, \eqref{ineq_fixedpointhypo} is found.

	Therefore, in view of the Leray-Schauder’s Fixed Point Theorem, we have that \eqref{adj:Stefanplain}
possesses at least one solution.

\
	Now, let us see that the solution we have found is unique.
	Let $(\varphi_1,\gamma_1)$ and~$(\varphi_2,\gamma_2)$ be two solutions (in~$H^{1,2}(Q)\times H^1(0,T)$) to~\eqref{adj:Stefanplain}.
	Let us set $\Phi=\varphi_1-\varphi_2$ and $\Gamma=\gamma_1-\gamma_2$.
	Then, $(\Phi,\Gamma)$ is a solution to
\begin{equation}\label{adj:Stefanplain_uniqueness}
\left\{
	\begin{array}{lll}
	\bar q(t)\Phi_t-\Phi_{xx}-a\Phi_x-b\Phi=0& \mbox{ in }&Q,\\
	\noalign{\smallskip}\displaystyle
	\Phi(-1,t)=0& \mbox{ in }&(0,T),\\
	\noalign{\smallskip}\displaystyle
	\Phi(1,t)=\Gamma(t)+(N(\cdot\,,t),\Phi(\cdot\,,t))_2& \mbox{ in }&(0,T),\\
	\noalign{\smallskip}\displaystyle
	\Phi(\cdot\,,T)=0& \mbox{ in }&(-1,1),\\ 
	\noalign{\smallskip}\displaystyle
	\Gamma'(t)=(R(\cdot\,,t),\Phi(\cdot\,,t))_2& \mbox{ in }&(0,T),\\
	\noalign{\smallskip}\displaystyle
	\Gamma(T)=0.
\end{array}
\right.
\end{equation}
	Proceeding as in the previous step, we obtain
\begin{equation}\label{est:regaux1_unique}
\begin{alignedat}{2}
	&\int_\tau^T\left(\|\Phi_{t}(\cdot\,,t)\|^2_{2}+\|\Phi(\cdot\,,t)\|^2_{H^2(-1,1)}\right)\,dt + \|\Phi(\cdot\,,\tau)\|_{H^1(-1,1)}^2\\
	&\quad \leq C\bigg(\int_\tau^T\left(1+\|a(\cdot\,,t)\|^2_\infty +\|b(\cdot\,,t)\|^2_\infty\right)\|\Phi(\cdot\,,t)\|^2_{H^1(-1,1)}\,dt\\
	&\qquad+ \int_\tau^T(|\Phi(1,t)|+|\Phi_{t}(1,t)|)|\Phi_x(1,t)|\,dt\bigg).
\end{alignedat}
\end{equation}
	Finally, there is no difficult to estimate  the boundary terms and apply Gronwall's inequality to deduce that $\Phi\equiv0$ and, consequently, $\Gamma\equiv0$.
	%%%%%%%%%%%%%%%%%%%
\section{Proof of Lemma~\ref{lemma:nonlocal}}\label{sec:app_C}

	For brevity, the Lebesgue integration elements $dx$ and $dt$ will be omitted in this section. On the other hand, $(\cdot\,,\cdot)$ and~$\|\cdot\|$ will stand for the usual scalar product and norm in~$L^2(Q)$.
	
	The main difficulties in the proof are that we have to work with non-local terms both in the time and the space variables.
	In order to deal with the nonlocal in time terms, we have started the computations using that the time derivatives do not exhibit nonlocal behavior in time.
	In addition, we will take advantage of the fact that the nonlocal in space terms are written on the boundary, at~$x=1$, just where $-\alpha$ and $\xi$ attain their respective minima. 

	We start by noting that
\begin{equation}\label{eq:alpha_deriv}
\begin{array}{l}\dis
	\alpha_x =-\lambda\xi\eta_x, \qquad
	\alpha_{xx}= -\lambda^2\xi\eta_x^2 -\lambda\xi\eta_{xx},\\% \qquad
	%\alpha_{xxx}= - \lambda^3\xi\eta_x^3 - 3\lambda^2\xi\eta_{xx}\eta_x-\lambda\xi\eta_{xxx},\\ 
	%\noalign{\smallskip} \dis
	%\alpha_{xxxx}= - \lambda^4\xi\eta_x^4 - 6\lambda^3\xi\eta_x^2\eta_{xx} - 3\lambda^2\xi\eta_{xx}^2-4\lambda^2\xi\eta_x\eta_{xxx}- \lambda\xi\eta_{xxxx}, \\ 
	 \noalign{\smallskip}\dis
	\alpha_t = -\xi^2\left[e^{-2\lambda\eta} - e^{-\lambda(m\|\eta\|_\infty+\eta)}\right](T-2t),\quad\ 
	\alpha_{xt} = \lambda\xi^2\eta_xe^{-\lambda(m\|\eta\|_\infty+\eta)}(T-2t) ,
	\\ \noalign{\smallskip} \dis
	\alpha_{tt}= 2\xi^2\left[e^{-2\lambda\eta}-e^{-\lambda(m\|\eta\|_\infty+\eta)}\right]+2(T-2t)^2\xi^3\left[e^{-\lambda (m\|\eta\|_\infty+3\eta)}-e^{-2\lambda(m\|\eta\|_\infty+\eta)}\right].
	\end{array}
\end{equation}
	It follows that there exists $C>0$ such that the following pointwise estimates are satisfied for sufficiently large $\lambda$ at any $(x,t) \in \overline{Q}$:
\begin{equation}\label{eq:alpha_esti}
\begin{split}
	%&|\partial^l_x\alpha|\leq~C\lambda^3\xi, \ \ \mbox{ for } 0\leq l\leq 4 , \\
	&|\alpha_t| \leq CT\xi^2,\ \ |\alpha_{xt}| \leq T\lambda\xi^2,\ \ 
	|\alpha_{tt}| \leq C(\xi^2+T^2\xi^3)\leq ~CT^2\xi^3.
\end{split}
\end{equation}
	
	We set $w=e^{-s\alpha}\psi$ and, using the boundary conditions satisfied by~$\psi$, we observe that $w(-1,\cdot)=0$.
	Also, note that, from the definitions of $\alpha$ and $w$, we get:
	$$
\lim_{t\to0^+}t^{-2}(T-t)^{-2}w(\cdot\,,t)=\lim_{t\to T^-}t^{-2}(T-t)^{-2}w(\cdot\,,t)=0
	$$
and
	$$
w_x(\cdot\,,T)=w_x(\cdot\,,0)=0.
	$$

	Let us introduce the partial differential operator
	$P:=\partial_t+d\partial_{xx}$.  We have the following decomposition
\[
e^{-s\alpha}f=e^{-s\alpha} P(e^{s\alpha} w ) = P_e\, w + P_k\,w ,
\]
	where $P_{e} w:=dw_{xx}+(s\alpha_t+s^2d\alpha_x^2)w$ is the self-adjoint part of~$P$ and $P_{k}w:=w_t+2sd\alpha_xw_x+sd \alpha_{xx}w$ is the skew-adjoint part.
	It follows that
	\begin{equation}
P_ew+(P_kw-sd \alpha_{xx}w)=e^{-s\alpha}f-sd\alpha_{xx}w 
	\end{equation}
and, consequently, 
	\begin{equation}\label{eq:L2norm}
%\begin{alignedat}{2}
\|e^{-s\alpha}f-sd\alpha_{xx}w\|^2=%&~
\|P_e w\|^2+\|P_kw-sd \alpha_{xx}w\|^2%\\ &
+2(P_e w,P_kw-sd \alpha_{xx}w).
%\end{alignedat}
	\end{equation}

	The rest of the proof is devoted to analyzing the term $(P_e w,P_kw-sd \alpha_{xx}w)$.
	Thus, from the above definition of the operators $P_e$ and $P_k$, it follows that
\begin{equation}\label{eq:pepk}
\begin{alignedat}{2}
	2(P_e w,P_kw-sd \alpha_{xx}w) &=	~2\left(dw_{xx}, w_t\right)+2\left(dw_{xx},2sd\alpha_x w_x\right)
						\\
						&+2\left(s\alpha_tw+s^2d\alpha_x^2w,w_t\right)
						+2\left(s\alpha_tw+s^2d \alpha _x^2w,2sd\alpha_x w_x\right)\\
	&=:~I_1 + I_2 + I_3 + I_4.
\end{alignedat}
\end{equation}
	
	For the first integral term $I_1$, we integrate by parts in space and obtain that
\begin{equation}\label{eq:I_1}
\begin{alignedat}{2}
	I_1=&\,- 2\jjnt_Q dw_{x}w_{xt}+2\int_0^T
	      \left[dw_tw_x\right]_{x=-1}^{x=1}\\
= &\,\jjnt_Q d_tw_{x}^2-\int_{-1}^1 \left[dw_{x}^2 \right]_{t=0}^{t=T}
+2 \int_0^T\left[dw_t w_x\right]_{x=-1}^{x=1}.
\end{alignedat}
\end{equation}
	For the second one, we integrate again by parts in space and deduce that
\begin{equation}\label{eq:I_2}
\begin{alignedat}{2}
	I_2
	=&~-2s\jjnt_Q d^2\alpha_{xx}|w_x|^2 +2s\int_0^T\left[d^2 \alpha_x|w_x|^2\right]_{x=-1}^{x=1}.
\end{alignedat}
\end{equation}
	For the third term, we integrate by parts in time.
	The following is found:
\begin{equation}\label{eq:I_4}
\begin{alignedat}{2}
	I_3=&~  -s\jjnt_Q\alpha_{tt} |w|^2-s^2\jjnt_Q (d\alpha_x^2)_t |w|^2 +\int_{-1}^1 \left[(s\alpha_t+s^2d\alpha_x^2)|w|^2\right]_{t=0}^{t=T}.
\end{alignedat}
\end{equation}
	Then, for the fourth term, we see that
\begin{equation}\label{eq:I_5}
	I_4= -\jjnt_Q d\left(2s^2(\alpha_t\alpha_{x})_x+6s^3d\alpha_x^2\alpha_{xx}\right)|w|^2+2\int_0^T\left[ d (s^2\alpha_t\alpha_x+s^3d\alpha_x^3)  |w|^2\right]_{x=-1}^{x=1}.
\end{equation}

	From \eqref{eq:pepk}-\eqref{eq:I_5}, we  get:
\[
\begin{alignedat}{2}
	2(P_e w,P_kw-sd \alpha_{xx}w) &=~\jjnt_Q (-2sd^2\alpha_{xx}+d_t)|w_x|^2 \\
	&+ \jjnt_Q\left(-s\alpha_{tt}-s^2(d\alpha_x^2)_t-2s^2d(\alpha_x\alpha_{t})_x-6s^3d^2\alpha_x^2\alpha_{xx} \right)|w|^2\\
	&+\int_0^T\!\!\!\!\left[2dw_tw_x+2sd^2\alpha_x |w_x|^2 + 2s^2d\alpha_x(\alpha_t+sd\alpha_x^2)|w|^2
	 \right]_{x=-1}^{x=1}\\
	&-\int_{-1}^1 \left[dw_{x}^2 -\left(s\alpha_t+s^2d\alpha_x^2\right)|w|^2\right]_{t=0}^{t=T}\\
	%&=~\,dt1+\,dt2+BT+I\,dt,
	&=~I_{D1}+I_{D2}+I_{BS}+I_{BT} ,
\end{alignedat}
\]
	where $I_{D1}$ and $I_{D2}$ correspond to distributed terms,  $I_{BS}$ is related to the boundary terms and $I_{BT}$ contains initial and final terms.
	Obviously, $I_{BT}=0$.
	
	Let us estimate the distributed terms.
	Thanks to~\eqref{eq:alpha_deriv} and from \eqref{P1}, we have
\begin{align*}
	I_{D1}=&~2s\lambda^2\jjnt_Q d^2\eta_x^2\xi |w_x|^2
	+2s\lambda\jjnt_Q d^2\eta_{xx}\xi |w_x|^2+\jjnt_Q d_t|w_x|^2\\
	\geq&~Cs\lambda^2\jjnt_Q \xi |w_x|^2 
		-Cs\lambda^2\iio \xi |w_x|^2 
	-C\left(s\lambda\jjnt_Q \xi |w_x|^2+\jjnt_Q |w_x|^2\right).
\end{align*}
	Hence,  using the fact that $(s\xi)^{-1}\leq 1/(4s_0)$ and $\lambda\geq \lambda_0$ and taking $s_0$ and $\lambda_0$ large enough, we obtain:
\begin{equation}\label{eq:dt1}
	Cs\lambda^2\iio \xi |w_x|^2 + I_{D1}\geq Cs\lambda^2\jjnt_Q \xi |w_x|^2.
\end{equation}

	Also, in order to get an estimate for $I_{D2}$, we use \eqref{P1}, \eqref{eq:alpha_deriv}, \eqref{eq:alpha_esti} and the fact that $s\geq s_0(T+T^2)$ and 
	$\lambda\geq\lambda_0$.
	This gives:
\begin{equation}\label{eq:dt2}
	 Cs^3\lambda^4\iio\xi^3 |w|^2+ I_{D2}
	\geq Cs^3\lambda^4\jjnt_Q\xi^3 |w|^2.
\end{equation}	

	Finally, let us estimate the integral containing boundary terms.
	Recalling that  $w(-1,\cdot)=0 $ in $(0,T)$, we deduce that:
\begin{equation}\label{eq:inner}
\begin{split}	
	I_{BS}&=2s^2\int_0^T d\alpha_x(\alpha_t+sd\alpha_x^2)|w|^2\big|_{x=1}
	+2s\int_0^T\left[d^2\alpha_x |w_x|^2\right]_{x=-1}^{x=1}+2\int_0^Tdw_tw_x\big|_{x=1}\\
	&=: I_{BS1}+I_{BS2}+I_{BS3}.
\end{split}
\end{equation}
	%where $BT_1$, $BT_2$ and $BT_3$ correspond to the boundary terms.
	Thanks to \eqref{eq:alpha_deriv}, \eqref{eq:alpha_esti} and the fact that $s\geq s_0(T+T^2)$ and $w_t=-s\alpha_tw+e^{-s\alpha}\psi_t$, we see that
\begin{align*}
	I_{BS1}\geq&-2s^3\lambda^3\int_0^Td^2\eta_x^3\widehat\xi^3 |w|^2\big|_{x=1} 
	-Cs^3\lambda\int_0^T\widehat\xi^3|w|^2\big|_{x=1} ,\\
	I_{BS2}=&-2s\lambda\int_0^T\left[d^2\eta_x\widehat\xi |w_x|^2\right]_{x=-1}^{x=1} ,\\
	%BT_3\geq&~2\int_0^Tw_tBw_x\big|^{x=1} -2s\int_0^TB^2\xi (\lambda^2\eta_x^2 +\lambda\eta_{xx}) ww_x \big|_{x=1}.
	I_{BS3}\geq&~2\int_0^Td\psi_tw_xe^{-s\widehat\alpha}\big|_{x=1}-Cs^3\int_0^T\widehat\xi^3 |w|^2 \big|_{x=1}-Cs\int_0^T\widehat\xi |w_x|^2 \big|_{x=1}.
\end{align*}	
	Using again~\eqref{P1}, the fact that $(s\xi)^{-1}\leq 1/(4s_0)$ and $\lambda\geq \lambda_0$, taking $s_0$ and $\lambda_0$ large enough and recalling the Cauchy-Schwarz inequality, we find from the previous estimate that
	\begin{equation}\label{eq:inner2}
	I_{BS}\geq C\int_0^T\!\left(s^3\lambda^3\widehat\xi^3 |w|^2 \!+\! s\lambda\widehat\xi |w_x|^2\right)\!\big|_{x=1} \!+\! Cs\lambda\int_0^T\!\widehat\xi |w_x|^2\big|_{x=-1}
\!-\! C\int_0^T(s\lambda\widehat\xi)^{-1}e^{-2s\widehat\alpha}|\psi_t|^2\big|_{x=1}.
	\end{equation}

	From \eqref{eq:L2norm},  \eqref{eq:dt1}, \eqref{eq:dt2} and \eqref{eq:inner2} and  the fact that $(s\xi)^{-1}\leq 1/(4s_0)$ and $\lambda\geq \lambda_0$, taking $s_0$ and $\lambda_0$ large enough, we conclude that
\begin{equation}\label{eq:concl1}
\begin{alignedat}{2}
	&\|P_e w\|^2+\|P_kw-sd \alpha_{xx}w\|^2\\
	&~~~+s^3\lambda^4\jjnt_Q\xi^3 |w|^2 \!+\! s\lambda^2\jjnt_Q \xi |w_x|^2 \!+\! \int_0^T\left(s^3\lambda^3\widehat\xi^3 |w|^2 \!+\! s\lambda\widehat\xi |w_x|^2\right)\big|_{x=1}
	\!+\! s\lambda\int_0^T\widehat\xi |w_x|^2\big|_{x=-1}\\
	& \leq  C\bigg(\|e^{-s\alpha} f\|^2_{2} 
	\!+\! s^3\lambda^4\iio\xi^3|w|^2 \!+\! s\lambda^2\iio \xi	|w_x|^2
	\!+\! \int_0^T(s\lambda\widehat\xi)^{-1}e^{-2s\widehat\alpha}|\psi_t|^2\big|_{x=1}\bigg).
\end{alignedat}
\end{equation}
	
	Now, using that $P_{e} w=w_{xx}+(s\alpha_t+s^2\alpha_x^2)w$, we get:
\begin{equation}\label{eq:laplac}
\begin{alignedat}{2}
	s^{-1}\jjnt_Q\xi^{-1} |w_{xx}|^2
	=&~ s^{-1}\jjnt_Q\xi^{-1}  |P_e w-(s\alpha_t+s^2\alpha_x^2)w|^2\\
	\leq&~ C s^{-1}\jjnt_Q\xi^{-1}\left(  |P_e w|^2+ s^2\lambda^2\xi^4|w|^2+s^4\lambda^4\xi^4|w|^2\right)\\
	\leq&~ C\left( s^{-1}\jjnt_Q\xi^{-1} |P_e w|^2+\jjnt_Q s^3\lambda^4\xi^3|w|^2\right).
\end{alignedat}
\end{equation}

	We can do the same for $P_{k}w-sd \alpha_{xx}w=w_t+2s\alpha_xw_x$.
	Then,
\begin{equation}\label{eq:time}
\begin{alignedat}{2}
	s^{-1}\jjnt_Q\xi^{-1} |w_t|^2
	=&~ s^{-1}\jjnt_Q\xi^{-1}  |(P_{k}w-sd \alpha_{xx}w)-2s\alpha_xw_x|^2\\
	\leq&~ C s^{-1}\jjnt_Q\xi^{-1}\left(  |P_{k}w-sd \alpha_{xx}w|^2+ s^2\lambda^2\xi^2|w_x|^2\right)\\
	\leq&~ C\left( s^{-1}\jjnt_Q\xi^{-1} |P_{k}w-sd \alpha_{xx}w|^2+\jjnt_Q s\lambda^2\xi |w_x|^2\right).
\end{alignedat}
\end{equation}

	From \eqref{eq:concl1}, \eqref{eq:laplac} and~\eqref{eq:time}, by introducing a cut-off function to estimate the local gradient integral and performing the usual integration by parts, the following holds
\begin{equation}\label{eq:concl4}
\begin{alignedat}{2}
	&~\jjnt_Q s^{-1}\xi^{-1} (|w_t|^2+ |w_{xx}|^2)+ \jjnt_Q s\lambda^2\xi |w_x|^2 
	+s^3\lambda^4\jjnt_Q\xi^3 |w|^2\\
	~&\ \ \ \ +s\lambda\int_0^T\widehat\xi |w_x|^2\big|_{x=-1}+\int_0^T\left(s^3\lambda^3\widehat\xi^3 |w|^2+s\lambda\widehat\xi |w_x|^2\right)\big|_{x=1} \\
	~ &\ \ \ \ \leq C\left(\|e^{-s\alpha} f\|^2_{L^2(Q)} + s^3\lambda^4\iil\xi^3|w|^2+s^{-1}\lambda^{-1}\int_0^T\xi^{-1}e^{-2s\alpha}|\psi_t|^2\big|_{x=1}\right).
\end{alignedat}
\end{equation}

	Notice that $w_x\big|_{x=-1}=e^{-s\widehat\alpha} \psi_x\big|_{x=-1}$ since $\psi(-1,\cdot)=0$ and $w_x\big|_{x=1}=e^{-s\widehat\alpha} \psi_x\big|_{x=1}+s\lambda\widehat\xi\eta_x w\big|_{x=1}$. Thus, we can come back to $\psi$ and deduce that
\begin{equation}\label{eq:concl2}
	I(s,\lambda,\psi)
	~\leq C\left(\jjnt_Q e^{-2s\alpha} |f|^2+s^3\lambda^4\iil e^{-2s\alpha} \xi^3 |\psi|^2
	+s^{-1}\lambda^{-1}\int_0^T\widehat\xi^{-1}e^{-2s\widehat\alpha} |\psi_t|^2\big|_{x=1}\right),
\end{equation}
where we have set
\begin{equation*}
\begin{split}
I(s,\lambda,\psi):=& \jjnt_Qe^{-2s\alpha}\left[ (s\xi)^{-1} (|\psi_t|^2+|\psi_{xx}|^2)+s\lambda^2\xi |\psi_x|^2+s^3\lambda^4\xi^3 |\psi|^2\right]+ s^3\lambda^3\int_0^Te^{-2s\widehat\alpha}\widehat\xi^3 |\psi|^2\big|_{x=1}\\
&+s\lambda\int_0^Te^{-2s\widehat\alpha}\widehat\xi |\psi_x|^2 \big|_{x=-1}
+s\lambda\int_0^Te^{-2s\widehat\alpha}\widehat\xi |\psi_x|^2 \big|_{x=1}.
\end{split}
\end{equation*}

	To conclude the proof, we have to eliminate the last term in~\eqref{eq:concl2}. 
	Using \eqref{eq:adj_1}$_{3,5}$, we find that
$$
	 \psi_t\big|_{x=1} =(R(\cdot\,,t)+N_t(\cdot\,,t),\psi(\cdot\,,t))_2+(N(\cdot\,,t),\psi_t(\cdot\,,t))_2
	 +g.
$$
	Then, using the fact that $R\in L^\infty(0,T;L^2(-1,1))$ and $N\in W^{1,\infty}(0,T;L^2(-1,1))$ and performing some immediate estimates, we obtain:
	\begin{equation}\label{eq:concl22}
\begin{alignedat}{2}
	\int_0^T(s\lambda\widehat\xi)^{-1}e^{-2s\widehat\alpha}|\psi_t|^2\big|_{x=1}\leq &~	C\jjnt_Q(s\lambda\widehat\xi)^{-1}e^{-2s\widehat\alpha}(|\psi|^2+|\psi_t|^2)
+\int_0^T(s\lambda\widehat\xi)^{-1}e^{-2s\widehat\alpha}|g|^2.
\end{alignedat}
\end{equation}

	Note that $\widehat\xi(t)^{-1} e^{-2s\widehat\alpha(t)}\leq \xi(x,t)^{-1} e^{-2s\alpha(x,t)}$ for all $(x,t)\in Q$.
	Accordingly, we have from \eqref{eq:concl22} that
\[
\begin{alignedat}{2}
	s^{-1}\lambda^{-1}\int_0^T\widehat\xi^{-1}e^{-2s\widehat\alpha}|\psi_t|^2\big|_{x=1}\leq &~	Cs^{-1}\lambda^{-1}\jjnt_Q\xi^{-1}e^{-2s\alpha}(|\psi|^2+|\psi_t|^2)
	+s^{-1}\lambda^{-1}\int_0^T\widehat\xi^{-1}e^{-2s\widehat\alpha}|g|^2\\
	\leq &~	C\lambda_0^{-1}s^{-1}\jjnt_Q\xi^{-1}e^{-2s\alpha}|\psi_t|^2+{ {C\over 256\,s_0^4\lambda_0^5}}s^3\lambda^4 \jjnt_Q\xi^{3}e^{-2s\alpha}|\psi|^2
	\\&+s^{-1}\lambda^{-1}\int_0^T\widehat\xi^{-1}e^{-2s\widehat\alpha}|g|^2 .
\end{alignedat}
\]
	This estimate, used together with \eqref{eq:concl2} and taking $s_0$ and $\lambda_0$ large enough, leads to \eqref{carleman:1}. This ends the proof.
%%%%%%%%%%%%%%%%%%%
\section{Proof of Proposition~\ref{prop:sol_adj_7}}\label{sec:app_B}

	The proof of existence can be achieved via the {\it Faedo-Galerkin method.}
	It will be divided into several steps.

{\sc 1. Galerkin approximations}
	
	Let $\{w_1, w_2, \dots \}$ be the ``special'' basis of $H^1_0(-1,1)$, formed by the eigenfunctions of the Dirichlet Laplacian, ortogonal in this space and orthonormal in $L^2(-1,1)$.
	For each $n\geq 1$, we will look for a pair 
	$(z_n,h_n):[0,T]\mapsto H^1_0(-1,1)\times \mathbb{R}$ with
\begin{equation}\label{eq:y_p_n}
	z_n(t)=\sum_{k=1}^na_n^k(t)w_k
\end{equation}
	satisfying
\begin{equation}\label{eq:z}
\left\{
		\begin{array}{ll}
			\bar q(z'_n,w_k)_2 \!+\! (z_{n,x}, w_{k,x})_2 \!+\! (a z_{n,x}+ h_nR,w_k)_2
			\!+\! z_{n,x}(1,\cdot)(N,w_k)_2 \!=\! (F,w_k)_2 \ \ \forall k = 1,\dots,n\\
			\noalign{\smallskip}\dis
			 h_n'(t) + z_{n,x}(1,t) = G,	\quad (0\leq t\leq T),
		\end{array}
		\right.
\end{equation}	
\begin{equation}\label{eq:y_p_0}
	a_n^k(0)=(z_0,w_k)\ \ (k=1,\ldots,n) \ \hbox{ and } \ h_n(0)=h_0.
\end{equation}

	Obviously, \eqref{eq:z}--\eqref{eq:y_p_0} is a Cauchy problem for a first order linear system of ODEs.
	Consequently, the existence and uniqueness of absolutely continuous functions~$a_n^k$ and~$h_n$ on~$[0,T]$ is ensured.

{\sc 2. A priori estimates}

	Now, the goal is to get some uniform estimates of the couples~$(z_n, h_n)$.
	To do this, let us multiply the first equation of \eqref{eq:z} by $a_n^k(t)$ and sum from $1$ to $n$.
	Let us also multiply the second equation by~$h_n$.
	The following is found in~$(0,T)$:
\[
	\left\{
		\begin{array}{ll}
			\bar q( z'_n,z_n)_2+(z_{n,x},z_{n,x})_2 
			+ (a z_{n,x},z_n)_2+h_n(R,z_n)_2
			+ z_{n,x}(1,\cdot)(N,z_n)_2=(F,z_n)_2, \\
			 \noalign{\smallskip}\dis
			 h_n'(t)h_n(t)+z_{n,x}(1,t)h_n(t)=Gh_n(t).
		\end{array}
		\right.
\]

	Summing these identities, we easily obtain:
\[
	\begin{array}{l} \dis
	\dis{d\over dt}\left(\bar q\| z_n(t)\|_2^2+|h_n|^2\right)+\|z_{n,x}\|_2^2\leq %%{\|\bar q'\|_\infty\over 2}+{\|A\|^2_{\infty}\over2} +\|B\|^2_\infty
	2|z_{n,x}(1,t)|^2+(|\bar q'|+\|a(\cdot\,,t)\|^2_2+\|R(\cdot\,,t)\|^2_2+\|N(\cdot\,,t)\|^2_2)\| z_n(t)\|^2_2
	\\
	 \noalign{\smallskip}\dis \dis \phantom{{1\over 2}\dis{d\over dt}\left(\bar q\| z_n(t)\|_2^2+\beta|h_n|^2\right)+{1\over 2}\|z_{n,x}\|_2^2}
	+3|h_n(t)|^2+\|F(\cdot\,,t)\|^2_2+|G(t)|^2.
	\end{array}
\]
	Let us also multiply the first $n$ identities of \eqref{eq:z} by the corresponding $a_{n,t}^k(t)$ and $\lambda_ka_{n}^k(t)$ and let us sum from $1$ to $n$.
	Then,
\[\left\{	\begin{array}{ll}
			\bar q\|z'_n\|^2_2+(z_{n,x},z'_{n,x})_2 
			+ (a z_{n,x},z'_n)_2+h_n(R,z'_n)_2
			+ z_{n,x}(1,\cdot)(N,z'_n)_2=(F,z'_n)_2,\\
			 \noalign{\smallskip}\dis
			\bar q( z'_{n,x},z_{n,x})_2 \!+\! \|z_{n,xx}\|^2 
			\!+\! (a z_{n,x},-z_{n,xx})_2 \!+\! h_n(R,-z_{n,xx})_2
			\!+\! z_{n,x}(1,\cdot)(N,-z_{n,xx})_2 \!=\! (F,-z_{n,xx})_2.
		\end{array}
		\right.
\]

	After summing, we see that 
\[
		\begin{alignedat}{2}
	&\dis{d\over dt}\left[\bar q(\| z_n(t)\|^2_2+\| z_{n,x}(t)\|^2_2)+\| z_{n,x}(t)\|^2_2+|h_n|^2\right]+(\|z_{n,x}(t)\|^2_2+\bar q\|z'_n(t)\|^2_2+\|z_{n,xx}(t)\|^2_2)\\
	 \noalign{\smallskip}\dis
	&\ \leq %%{\|\bar q'\|_\infty\over 2}+{\|A\|^2_{\infty}\over2} +\|B\|^2_\infty
	C(1+|\bar q'|+\|a(\cdot\,,t)\|^2_2+\|R(\cdot\,,t)\|^2_2+\|N(\cdot\,,t)\|^2_2)(\| z_n(t)\|^2_{H^1_0(-1,1)}+|h_n(t)|^2)\\
	&\quad +C(1+\|N(\cdot\,,t)\|^2_2)|z_{n,x}(1,t)|^2+C(\|F(\cdot\,,t)\|^2_2+|G(t)|^2)\\
	 \noalign{\smallskip}\dis
	&\ \leq %%{\|\bar q'\|_\infty\over 2}+{\|A\|^2_{\infty}\over2} +\|B\|^2_\infty
	\delta \|z_{n,xx}(t)\|^2_2+C(\|F(\cdot\,,t)\|^2_2+|G(t)|^2)\\
	&\quad +C(1+|\bar q'|+\|a(\cdot\,,t)\|^2_2+\|R(\cdot\,,t)\|^2_2+\|N(\cdot\,,t)\|^{16}_2+\|N(\cdot\,,t)\|^2_2)(\| z_n(t)\|^2_{H^1_0(-1,1)}+|h_n(t)|^2),
		\end{alignedat}
\]
	whence
\[
		\begin{alignedat}{2}
	&\dis{d\over dt}\left[\bar q(\| z_n(t)\|^2_2+\| z_{n,x}(t)\|^2_2)+\| z_{n,x}(t)\|^2_2+|h_n|^2\right]+(\|z_{n,x}(t)\|^2_2+\bar q\|z'_n(t)\|^2_2+\|z_{n,xx}(t)\|_2^2)\\
		&\leq %%{\|\bar q'\|_\infty\over 2}+{\|A\|^2_{\infty}\over2} +\|B\|^2_\infty
	C(\|F(\cdot\,,t)\|_2^2+|G(t)|^2) \\
	&\quad +C(1+|\bar q'|+\|a(\cdot\,,t)\|^2_2+\|R(\cdot\,,t)\|^2_2+\|N(\cdot\,,t)\|^{16}_2+\|N(\cdot\,,t)\|^2_2)(\| z_n(t)\|^2_{H^1_0(-1,1)}+|h_n(t)|^2).
		\end{alignedat}
\]

	Then, from  Gronwall's Lemma, we deduce that
\begin{equation}\label{est_z_infty}
	\| z_n\|^2_{L^\infty(0,T;H^1_0(-1,1))}+|h_n|^2_{L^\infty(0,T)}\leq e^{C(1+T)}
	\left(\|z_0\|_{H_0^1(-1,1)}^2+\|f_1\|^2_{L^2(Q)}+\|f_2\|^2_{L^2(0,T)}\right).
\end{equation}
	Therefore, one has:
\begin{equation}\label{est_z_infty_total}
	\| z_n\|^2_{H^{1,2}_{0}(Q)}+|h_n|^2_{H^1(0,T)}\leq e^{C(1+T)}\left(\|z_0\|_{H_0^1(-1,1)}^2+\|F\|^2_{L^2(Q)}+\|G\|^2_{L^2(0,T)}\right).
\end{equation}

{\sc 3. The existence of a strong solution}

	Let us take limits in (a subsequence of) the sequence $(z_n,h_n)$.
	
	In view of the a priori estimates \eqref{est_z_infty} and \eqref{est_z_infty_total}, there exists a subsequence (again indexed by~$n$) and functions $h\in L^2(0,T)$ and $z\in H^{1,2}_{0}(Q)$ such that
\[
	\begin{alignedat}{2}
			z_{n}\to z &\quad \hbox{weakly in}\quad L^2(0,T;H_0^1(-1,1)\cap H^2(-1,1)),\\
			z_{n,x}(1,\cdot)\to z_x(1,\cdot) &\quad \hbox{weakly in}\quad L^2(0,T),\\
			z_{n}'\to z_t& \quad \hbox{weakly in}\quad L^2(Q),\\
			h_{n}\to h& \quad \hbox{weakly in}\quad H^1(0,T).
	\end{alignedat}
\]

	Then, following standard and well known arguments, we can deduce that $(z,h)$ satisfies \eqref{eq:conStefantransnulllinear1}$_1$,  \eqref{eq:conStefantransnulllinear1}$_2$,  \eqref{eq:conStefantransnulllinear1}$_3$, \eqref{eq:conStefantransnulllinear1}$_5$ and~\eqref{eq:conStefantransnulllinear1}$_6$.

{\sc 4. Checking the initial conditions} 

	Thanks to the well known {\it Aubin-Lions' Lemma,} we have that $H^{1,2}_{0}(Q)$ is compactly embedded in $C^0([0,T]; L^2(-1,1))$. 
	Then, since $z_{n}\to z$ weakly in $H^{1,2}_{0}(Q)$, we also have that
	$$
z(\cdot\,,0)=\dis\lim_{n\to\infty}z_{n}(\cdot\,,0)=z_0.
	$$
	Similarly, since $h_{n}$ converges weakly in~$H^1(0,T)$, we deduce that
	$$
h(0)=\dis\lim_{n\to\infty}h_{n}(0)=h_0.
	$$

\

	The uniqueness of the solution is an almost direct consequence of energy estimates.
	Indeed, let $(z_1,h_1)$ and~$(z_2,h_2)$ be two solutions (in $H^{1,2}_{0}(Q)\times H^1(0,T)$) to~\eqref{eq:conStefantransnulllinear1} and let us set $z=z_1-z_2$ and $h=h_1-h_2$.
	Then, $(z,h)$ is a solution to
\[
\left\{
	\begin{array}{lcl}
		\bar q z_t- z_{xx}+az_x+ Rh+Nz_x(1,\cdot)=0	& \mbox{ in }&Q,\\
 		z(-1,\cdot)= 0								& \mbox{ in }&(0,T),\\
		z(1,\cdot)=0								& \mbox{ in }&(0,T),\\
		z(\cdot\,,0)= 0							& \mbox{ in }&(-1,1),\\
		h_t+ z_x(1,\cdot)=0					& \mbox{ in }& (0,T),\\
		h(0)= 0&&
	\end{array}
	\right.
\]
	and, from well known arguments, this shows that $z \equiv 0$ and~$h \equiv 0$.

\end{appendices}

\bibliographystyle{alpha}
\nocite{*}
%\bibliography{Stefantwophas}

\begin{thebibliography}{BFCMRM08}

\bibitem[AFCS21]{kazan}
R.~K.~C. Ara{\'u}jo, E.~Fern{\'a}ndez-Cara, and D.~A. Souza.
\newblock Remarks on the control of two-phase stefan free-boundary problems.
\newblock {\em submitted}, 2021.

\bibitem[ATF87]{alekseevoptimal}
V.~M. Alekseev, V.~M. Tikhomirov, and S.V. Fomin.
\newblock {\em Optimal control. Translated from the Russian by VM Volosov}.
\newblock Contemporary Soviet Mathematics, Consultants Bureau, New York, 1987.

\bibitem[Bar21]{barbu2021boundary}
V.~Barbu.
\newblock Boundary controllability of phase-transition region of a two-phase
  stefan problem.
\newblock {\em Systems Control Lett.}, 150:104896, 2021.

\bibitem[BFCMRM08]{brandao_cara}
A.~J.~V. Brand{\~a}o, E.~Fern{\'a}ndez-Cara, P.~M.~D. Magalh{\~a}es, and M.~A.
  Rojas-Medar.
\newblock Theoretical analysis and control results for the fitzhugh-nagumo
  equation.
\newblock {\em Electron. J. Differential Equations}, 2008(164):1--20, 2008.

\bibitem[BGT19]{boulakia2019well}
M.~Boulakia, S.~Guerrero, and T.~Takahashi.
\newblock Well-posedness for the coupling between a viscous incompressible
  fluid and an elastic structure.
\newblock {\em Nonlinearity}, 32(10):3548, 2019.

\bibitem[DE18]{darde2018reachable}
J.~Dard{\'e} and S.~Ervedoza.
\newblock On the reachable set for the one-dimensional heat equation.
\newblock {\em SIAM J. Control Optim.}, 56(3):1692--1715, 2018.

\bibitem[DFC05]{MR2139944}
A.~Doubova and E.~Fern\'{a}ndez-Cara.
\newblock Some control results for simplified one-dimensional models of
  fluid-solid interaction.
\newblock {\em Math. Models Methods Appl. Sci.}, 15(5):783--824, 2005.

\bibitem[DFC18]{MR3772848}
R.~Demarque and E.~Fern\'{a}ndez-Cara.
\newblock Local null controllability of one-phase {S}tefan problems in 2{D}
  star-shaped domains.
\newblock {\em J. Evol. Equ.}, 18(1):245--261, 2018.

\bibitem[DNPV12]{di2012hitchhiker}
E.~Di~Nezza, G.~Palatucci, and E.~Valdinoci.
\newblock Hitchhiker's guide to the fractional sobolev spaces.
\newblock {\em B. Sci. Math.}, 136(5):521--573, 2012.

\bibitem[FCdS17a]{MR3654149}
E.~Fern\'{a}ndez-Cara and I.~T. de~Sousa.
\newblock Local null controllability of a free-boundary problem for the
  semilinear 1{D} heat equation.
\newblock {\em Bull. Braz. Math. Soc. (N.S.)}, 48(2):303--315, 2017.

\bibitem[FCDS17b]{MR3736685}
E.~Fern\'{a}ndez-Cara and I.~T. De~Sousa.
\newblock Local null controllability of a free-boundary problem for the viscous
  {B}urgers equation.
\newblock {\em SeMA J.}, 74(4):411--427, 2017.

\bibitem[FCHL19]{MR3993192}
E.~Fern\'{a}ndez-Cara, F.~Hern\'{a}ndez, and J.~L\'{\i}maco.
\newblock Local null controllability of a 1{D} {S}tefan problem.
\newblock {\em Bull. Braz. Math. Soc. (N.S.)}, 50(3):745--769, 2019.

\bibitem[FCLdM16]{MR3433238}
E.~Fern\'{a}ndez-Cara, J.~Limaco, and S.~B. de~Menezes.
\newblock On the controllability of a free-boundary problem for the 1{D} heat
  equation.
\newblock {\em Systems Control Lett.}, 87:29--35, 2016.

\bibitem[FCZ00]{MR1750109}
E.~Fern\'{a}ndez-Cara and E.~Zuazua.
\newblock The cost of approximate controllability for heat equations: the
  linear case.
\newblock {\em Adv. Differential Equations}, 5(4-6):465--514, 2000.

\bibitem[FI96]{MR1406566}
A.~V. Fursikov and O.~Yu. Imanuvilov.
\newblock {\em Controllability of evolution equations}, volume~34 of {\em
  Lecture Notes Series}.
\newblock Seoul National University, Research Institute of Mathematics, Global
  Analysis Research Center, Seoul, 1996.

\bibitem[FPZ95]{MR1318622}
C.~Fabre, J.-P. Puel, and E.~Zuazua.
\newblock Approximate controllability of the semilinear heat equation.
\newblock {\em Proc. Roy. Soc. Edinburgh Sect. A}, 125(1):31--61, 1995.

\bibitem[FR71]{MR335014}
H.~O. Fattorini and D.~L. Russell.
\newblock Exact controllability theorems for linear parabolic equations in one
  space dimension.
\newblock {\em Arch. Rational Mech. Anal.}, 43:272--292, 1971.

\bibitem[FR99]{MR1684873}
A.~Friedman and F.~Reitich.
\newblock Analysis of a mathematical model for the growth of tumors.
\newblock {\em J. Math. Biol.}, 38(3):262--284, 1999.

\bibitem[Gup03]{gupta}
S.~C. Gupta.
\newblock {\em The classical {S}tefan problem}, volume~45 of {\em North-Holland
  Series in Applied Mathematics and Mechanics}.
\newblock Elsevier Science B.V., Amsterdam, 2003.
\newblock Basic concepts, modelling and analysis.

\bibitem[GZ21]{MR4253800}
B.~Geshkovski and E.~Zuazua.
\newblock Controllability of one-dimensional viscous free boundary flows.
\newblock {\em SIAM J. Control Optim.}, 59(3):1830--1850, 2021.

\bibitem[HKT20]{hartmann2020reachable}
A.~Hartmann, K.~Kellay, and M.~Tucsnak.
\newblock {From the reachable space of the heat equation to Hilbert spaces of
  holomorphic functions}.
\newblock {\em J. Eur. Math. Soc.}, 22(10):3417--3440, 2020.

\bibitem[KK20]{MR4032314}
S.~Koga and M.~Krstic.
\newblock Single-boundary control of the two-phase {S}tefan system.
\newblock {\em Systems Control Lett.}, 135:104573, 9, 2020.

\bibitem[Leg05]{Legates2005}
D.~R. Legates.
\newblock {\em Latent Heat}, pages 450--451.
\newblock Springer Netherlands, Dordrecht, 2005.

\bibitem[Lio88]{lions1988controlabilite}
J.~L. Lions.
\newblock {Contr{\^o}labilit{\'e} exacte, perturbations et stabilisation de
  systemes distribu{\'e}s, tome 1, RMA 8}, 1988.

\bibitem[LLW13]{MR2997373}
C.~Lei, Z.~Lin, and H.~Wang.
\newblock The free boundary problem describing information diffusion in online
  social networks.
\newblock {\em J. Differential Equations}, 254(3):1326--1341, 2013.

\bibitem[LR95]{MR1312710}
G.~Lebeau and L.~Robbiano.
\newblock Contr\^{o}le exact de l'\'{e}quation de la chaleur.
\newblock {\em Comm. Partial Differential Equations}, 20(1-2):335--356, 1995.

\bibitem[LSTY83]{saguez}
M.~Larrecq, C.~Saguez, V.~C. Tran, and J.~P. Yvon.
\newblock Optimal control of a continuous casting.
\newblock {\em IFAC Proceedings Volumes}, 16(10):218--223, 1983.

\bibitem[LTT13]{MR3023058}
Y.~Liu, T.~Takahashi, and M.~Tucsnak.
\newblock Single input controllability of a simplified fluid-structure
  interaction model.
\newblock {\em ESAIM Control Optim. Calc. Var.}, 19(1):20--42, 2013.

\bibitem[MRR16]{martin2016reachable}
P.~Martin, L.~Rosier, and P.~Rouchon.
\newblock On the reachable states for the boundary control of the heat
  equation.
\newblock {\em Appl. Math. Research eXpress}, 2016(2):181--216, 2016.

\bibitem[Ors21]{orsoni2021reachable}
M.-A. Orsoni.
\newblock {Reachable states and holomorphic function spaces for the 1-D heat
  equation}.
\newblock {\em J. Funct. Anal.}, 280(7):108852, 2021.

\bibitem[Str67]{strichartz1967multipliers}
R.~S. Strichartz.
\newblock Multipliers on fractional sobolev spaces.
\newblock {\em J. Math. Mech.}, 16(9):1031--1060, 1967.

\bibitem[WLL22]{MR4317441}
L.~Wang, Y.~Lan, and P.~Lei.
\newblock Local null controllability of a free-boundary problem for the
  quasi-linear 1{D} parabolic equation.
\newblock {\em J. Math. Anal. Appl.}, 506(2):Paper No. 125676, 26, 2022.

\bibitem[Zei86]{Zeidler}
E.~Zeidler.
\newblock {\em Nonlinear functional analysis and its applications. {I}}.
\newblock Springer-Verlag, New York, 1986.
\newblock Fixed-point theorems, Translated from the German by Peter R. Wadsack.

\end{thebibliography}

\end{document}